\documentclass{amsart}

\usepackage[latin1]{inputenc}
\usepackage[T1]{fontenc}
\usepackage[english]{babel}
\usepackage{soul}
\usepackage{amsmath}
\usepackage{amssymb}
\usepackage{mathrsfs}
\usepackage{amsthm}
\usepackage{bm}
\usepackage[all]{xy}
\usepackage{color}
\usepackage[a4paper, margin=3cm]{geometry}
\usepackage{bbold}
\usepackage{graphicx}
\usepackage{caption}
\usepackage{overpic}
\usepackage{enumitem}
\usepackage{calc}
\usepackage{xcolor}
\usepackage[colorinlistoftodos,textwidth=2cm]{todonotes}
\usepackage{soul}

\newtheorem{theorem}{Theorem}[section]

\newtheorem{lemma}[theorem]{Lemma}
\newtheorem{proposition}[theorem]{Proposition}

\theoremstyle{definition}
\newtheorem{definition}[theorem]{Definition}
\newtheorem{example}[theorem]{Example}
\newtheorem{remark}[theorem]{Remark}
\newtheorem{notation}[theorem]{Notation}

\newcommand{\Arf}{\mathit{Arf}}
\newcommand{\Z}{\mathbb{Z}}
\newcommand{\lk}{\ell k}
\newcommand{\rad}{\operatorname{rad}}
\newcommand{\T}{\operatorname{T}}

\newcommand{\dc}{\textcolor{black}}

\title{On Arf invariants of colored links}

\author{David Cimasoni}
\address{David Cimasoni -- Section de math\'ematiques, Universit\'e de Gen\`eve, Suisse}
\email{david.cimasoni@unige.ch}

\author{Ga\"etan Simian}
\address{Ga\"etan Simian -- Section de math\'ematiques, Universit\'e de Gen\`eve, Suisse}
\email{gaetan.simian@unige.ch}

\date{}

\begin{document}

\maketitle

\begin{abstract}
Several classical knot invariants, such as the Alexander polynomial, the Levine-Tristram signature and
the Blanchfield pairing, admit natural extensions from knots to links, and more generally, from oriented links to so-called colored links.
In this note, we explore such extensions of the Arf invariant. Inspired by the three examples stated above, we use
generalized Seifert forms to construct quadratic forms, and determine when the Arf invariant of such a form yields a well-defined invariant of colored links. However, apart from the known case of oriented links, these new Arf invariants turn out to be determined by the linking numbers.
\end{abstract}

\section{Introduction}
\label{sec:intro}

The Arf invariant of a knot~$K$ in~$S^3$ is among the simplest and best understood knot invariants.
It can be defined as follows. Let~$F$ be a {\em Seifert surface\/} for~$K$, i.e.\ an oriented connected compact
surface smoothly embedded in~$S^3$ with oriented boundary~$\partial F=K$. Let~$q\colon H_1(F;\Z_2)\to\Z_2$ be the map given by~$q(x)=\lk(x^-,x)$, where~$\lk$ stands for the linking number (here, reduced modulo~2), and~$x^-\subset S^3\setminus F$ denotes the cycle~$x\subset F$ pushed in the negative normal direction off~$F$. Then, the map~$q$ takes one of the two values in~$\Z_2$ more often than the other, and this value
does not depend on the choice of the Seifert surface~$F$. It is therefore an invariant of the knot~$K$, usually denoted by~$\Arf(K)\in\Z_2$, and referred to as the {\em Arf invariant\/} (or {\em Robertello invariant\/}~\cite{Rob65}) of~$K$.
Extending it from knots to oriented links is straightforward as long as the link~$L$ is such that~$\lk(K,L\setminus K)$ is even for each component~$K$ of~$L$~\cite{Rob65}. Indeed, this condition ensures that~$q$ induces a non-singular quadratic form,
whose Arf invariant (in the sense of the original article of Arf~\cite{Arf41}) can then be used,
see Section~\ref{sub:Arf-L} below.

The Arf invariant of knots enjoys a number of remarkable properties. First of all, it is an invariant of  knot concordance~\cite{Rob65}; more precisely, a knot has vanishing Arf invariant if and only if it is 0-solvable
in the sense of Cochran-Orr-Teichner~\cite{COT03}.
Also, it can be retrieved from the Alexander polynomial of~$K$ evaluated at~$t=-1$, as it coincides with the parity of the second coefficient of the Conway polynomial of~$K$~\cite{Rob65,Lev66,Mur69}. Furthermore, it can be computed from the Jones polynomial evaluated at~$t=i$, see e.g.~\cite[Chapter~10]{Lic97}.
Finally, it admits a number of alternative interpretations, for example in terms of the cobordism class
of the~0-surgery of~$K$~\cite[Remark~8.2]{COT03} (see also~\cite[Section~5.2]{C-N20} for details),
and in terms of twisted Whitney towers~\cite[Lemma~10]{CST14} (see also~\cite[Proposition~6.4]{FMN21}).

\medskip

As mentioned above, the Arf invariant of an oriented link can be defined and computed using the Seifert form~\cite{Sei35}, one of the most ubiquitous objects in classical knot theory. Other examples of invariants which can be defined, studied and computed using this tool include the Alexander-Conway polynomial~\cite{Ale28,Con70,Kau81}, the Alexander module~\cite{Sei35,Lev66}, the Levine-Tristram signature~\cite{Tro62,Lev69,Tri69}, and the Blanchfield form~\cite{Bla57,Kea75,F-P17}.

These four invariants can be naturally extended from oriented links to so-called {\em colored links\/}, whose definition we now recall.
An~$m${\em-colored link} is an oriented link~$L$ each of whose components is endowed with a {\em color\/} in~$\{1,\dots,m\}$ so that all colors are used.
Such a colored link is commonly denoted by~$L=L_1\cup\dots\cup L_m$,
with~$L_i$ the sublink of~$L$ consisting of the components of color~$i$.
Obviously, a~$1$-colored link is nothing but an oriented link,
while an~$m$-component~$m$-colored link is an oriented ordered link.

Let us now be more precise about the aforementioned extensions: for~$m$-colored links,
one can define~$m$-variable versions of the Alexander-Conway polynomial~\cite{Con70,Har83},
of the Alexander module~\cite{Ale28}, of the Levine-Tristram signature~\cite{C-F08}, and of the Blanchfield form~\cite{Bla57}.
Moreover, each of these invariants can be defined and studied using generalized Seifert surfaces 
known as {\em C-complexes\/}, first introduced by Cooper~\cite{Coo82} for~2-component links
and extended to general colored links in~\cite{Cim04}, see Sections~\ref{sub:C-complex} and~\ref{sub:GSM} below.
We refer the interested reader to~\cite{Cim04}
for such a geometric construction of the Alexander-Conway multivariable polynomial, to~\cite{C-F08}
for the Alexander module and Levine-Tristram signature, and to~\cite{CFT18,Con18} for the Blanchfield pairing.

\medskip

This work is an attempt to construct generalized Arf invariants of colored links using similar techniques.
One the one hand, this attempt is rather successful, as testified by our main result (see \dc{Section~\ref{sub:main}
for the construction}, Theorem~\ref{thm:main}
for a more complete version, together with Remarks~\ref{rem:Eodd} and~\ref{rem:c-ex} showing that all the
hypotheses are necessary).

\begin{theorem}
\label{thm:intro}
Let~$L$ be an~$m$-colored link, and let~$F$ be a C-complex for~$L$.
For each~$E\subset\{\pm\}^m$ of even cardinality,
there is a quadratic form~$q_F^E\colon H_1(F;\Z_2)\to\Z_2$
such that one of the following statements holds, the same one for any choice of~$F$:
\begin{itemize}
\item[(i)] $q_F^E$ induces a non-singular quadratic form, with Arf invariant~$0$;
\item[(ii)] $q_F^E$ induces a non-singular quadratic form, with Arf invariant~$1$;
\item[(iii)]  $q_F^E$ does not induce a non-singular quadratic form.
\end{itemize}
\end{theorem}

As will be made clear in Section~\ref{sub:main}, the classical Arf invariant corresponds to the case~$m=1$ and the choice of a sign, i.e.\ of an odd subset~$E\subset\{\pm\}$. However, for~$m>1$, only even subsets~$E\subset\{\pm\}^m$ yield invariants (see Remark~\ref{rem:Eodd}). By Theorem~\ref{thm:intro},
any such~$E\subset\{\pm\}^m$ of even cardinality defines an Arf invariant
for some class of~$m$-colored links which depends on~$E$.

On the other hand, our attempt can be considered as unsuccessful, due to the following fact.

\begin{proposition}
\label{prop:intro}
For any~$m$-colored link~$L$ and any~$E\subset\{\pm\}^m$  of even cardinality,
the trichotomy of Theorem~\ref{thm:intro} is determined by the linking numbers of the components of~$L$.
\end{proposition}

In conclusion, this note should serve as a cautionary tale for anyone interested in extending the Arf invariant in this manner.

\medskip

This paper is organized as follows. In Section~\ref{sec:background}, we recall the necessary background:
the Arf invariant of a quadratic form, the Arf invariant of a link, C-complexes and generalized Seifert matrices are dealt with in individual paragraphs.
Section~\ref{sec:main} contains the proof of Theorem~\ref{thm:intro},
and Section~\ref{sec:lk} the proof of Proposition~\ref{prop:intro}.
\dc{Finally, Appendix~\ref{sec:appendix} deals with generalized S-equivalence.}

\subsection*{Acknowledgements} The authors thank Anthony Conway and Stefan Friedl for helpful remarks on an earlier version of this article, and Christopher Davis for suggesting Lemma~\ref{lemma:classification}.
Support from the Swiss NSF grant 200021-212085 is thankfully acknowledged.

\section{Background on Arf invariants and generalized Seifert forms}
\label{sec:background}

This section contains all the background needed for this work:
we review the Arf invariant of a quadratic form  in Section~\ref{sub:Arf-q},
the Arf invariant of an oriented link in Section~\ref{sub:Arf-L}, C-complexes in
Section~\ref{sub:C-complex}, and the associated generalized Seifert matrices in Section~\ref{sub:GSM}.

\subsection{The Arf invariant of a quadratic form}
\label{sub:Arf-q}

In this section and the following one, we briefly recall the definition of the Arf invariant of a quadratic form~\cite{Arf41} and of an oriented link~\cite{Rob65}, following~\cite[Chapter~10]{Lic97} and referring to it for proofs and details.

\medskip

Let~$V$ be a finite-dimensional~$\Z_2$-vector space. A function~$q\colon V\to\mathbb{Z}_2$ is a {\em quadratic form\/} if there exists a bilinear map $B\colon V\times V \to \mathbb{Z}_2$ such that~$q(x+y)+q(x)+q(y)=B(x,y)$ for all $x,y\in V$. The quadratic form~$q$ is called {\em non-singular\/} if the bilinear map~$B$ itself is non-singular, i.e.\ if its radical
\[
\rad(B)=\{x\in V\,|\,B(x,y)=0\text{ for all~$y\in V$}\}
\]
is trivial.
In this case, since the non-singular bilinear map~$B$ also satisfies~$B(x,x)=0$ for all~$x\in V$, it is
by definition a symplectic form. Such forms are known to exist only if~$V$ has even dimension, say~$2n$.
It can then be shown that, as $x$ varies over the~$2^{2n}$ elements of~$V$,~$q(x)$ takes one of the values,~$0$ or~$1$, more often (namely~$2^{2n-1}+2^{n-1}$ times), than the other one ($2^{2n-1}-2^{n-1}$ times). This leads to 
what is sometimes referred to as the ``democratic invariant'' of~$q$.

\begin{definition}%Arf_quadratic
The {\em Arf invariant}  $\Arf(q)\in\Z_2$ of a non-singular quadratic form~$q$ is the value  taken the most often by $q(x)$ as~$x$ varies over the elements of~$V$. 
\end{definition}

A closed formula defining the Arf invariant is easily seen to be given by
\[
(-1)^{\Arf(q)}=\frac{1}{\sqrt{|V|}}\sum_{x\in V}(-1)^{q(x)}\,.
\]
A more practical formula, due to Cahit Arf\footnote{This formula appears on the reverse side of the 2009 Turkish 10 lira note.}, uses  a {\em symplectic basis\/} of~$V$, i.e.\ a basis ${e_1,f_1,\ldots,e_n,f_n}$ such that~$B(e_i,e_j)=B(f_i,f_j)=0$ and~$B(e_i,f_j)=\delta_{ij}$ for all~$1\le i,j\le n$. It reads:
\begin{equation}
\label{formularf}
\Arf(q)=\sum_{i=1}^n{q(e_i)q(f_i)}\,.
\end{equation}
One of the main results of Arf in~\cite{Arf41} is that two non-singular quadratic forms are isomorphic if and only if they have
same dimension and same Arf invariant.

This invariant can be extended to a possibly singular quadratic form~$q\colon V\to\mathbb{Z}_2$ as long as it vanishes on the radical
of the associated bilinear form~$B$. In such a case indeed, the map~$q\colon V\to\mathbb{Z}_2$ induces a well-defined non-singular quadratic form~$\overline{q}\colon V/\rad(B)\to\mathbb{Z}_2$, whose Arf invariant can be considered. On the other hand,
if~$q$ does not vanish on~$\rad(B)$, then its Arf invariant is not defined.\footnote{In particular, one easily checks that the voting process results in a draw.}

In summary, an arbitrary quadratic form~$q\colon V\to\Z_2$ always
satisfies the following trichotomy: either it induces a non-singular quadratic form with Arf invariant~$0$, or
it induces a non-singular quadratic form with Arf invariant~$1$, or it does not induce a non-singular quadratic form.
We shall simply denote these three possibilities by~$\Arf(q)=0$,~$\Arf(q)=1$, and~$\Arf(q)$ {\em undefined\/},
and call these three cases the three possible {\em values\/} for~$\Arf(q)$.

\subsection{The Arf invariant of an oriented link}
\label{sub:Arf-L}

As already mentioned in the introduction, this theory can be applied to knots and links, in the following way. Let~$L$ be an oriented link in~$S^3$, and let~$F$ be a Seifert surface for~$L$. Let~$q\colon H_1(F;\Z_2)\to\Z_2$
be the associated Seifert form reduced modulo~2 and restricted to the diagonal. In other words, the map~$q$ is defined by~$q(x)=\lk(x^-,x)$, where~$\lk$ stands for the linking number reduced modulo~2 and~$x^-\subset S^3\setminus F$ denotes the cycle~$x\subset F$ pushed  in
the negative direction off the oriented surface~$F$. For any~$x,y\in H_1(F;\Z_2)$, we have
\[
q(x+y)+q(x)+q(y)=\lk(x^-,y)+\lk(y^-,x)=\lk(x^-,y)+\lk(x^+,y)=B(x,y)\,,
\]
with~$B\colon H_1(F;\mathbb{Z}_2)\times H_1(F;\mathbb{Z}_2)\to\Z_2$ the modulo~2 intersection form. Therefore,
we see that the map~$q$ is a quadratic form.

This form is singular unless~$L$ is a knot, as~$\rad(B)$ is generated by
the cycles given by the boundary components of~$F$, i.e.\ the components of~$L$.
For~$q$ to induce a non-singular quadratic form~$\overline{q}$ on the quotient~$H_1(F;\Z_2)/\rad(B)$,
we therefore need~$q$ to vanish on the components of~$L$. This is the case if and only if~$\lk(K,L\setminus K)$
is even for each component~$K$ of~$L$.

\begin{definition}%Arf_link
Let~$L$ be an oriented link with~$\lk(K,L\setminus K)$ even for each component~$K\subset L$.
The {\em Arf invariant} $\Arf(L)\in\Z_2$ of~$L$ is the Arf invariant of the induced non-singular quadratic
form~$\overline{q}$.
\end{definition}

To prove that this is a well-defined invariant, one needs to check that it does not depend on the choice
of the Seifert surface~$F$ for~$L$. By a classical result (see e.g.~\cite[Chapter~8]{Lic97}),
two Seifert surfaces for the same link are connected by a finite sequence of ambient isotopies and handle attachments. This translates into an equivalence relation on Seifert matrices, known as {\em S-equivalence\/}, 
and it is an easy exercise to check that~S-equivalent matrices yield the same Arf invariant.

\subsection{C-complexes for colored links}
\label{sub:C-complex}

The goal of this section is to introduce generalized Seifert surfaces for colored links.
We start by recalling the definition of a colored link already sketched in the introduction. 

\begin{definition}%colored_link
Let~$m$ be a positive integer. An {\em ${m}$-colored link} is an oriented link $L$ in $S^3$ together with a surjective map assigning to each component of~$L$ a {\em color\/} in~$\{1,\ldots,m\}$.
Two colored links~$L$ and~$L'$ are said to be {\em isotopic\/} if there exists an isotopy between $L$ and $L'$ preserving the orientation and color of each component.
\end{definition}

We denote a colored link by~$L=L_1\cup\dots\cup L_m$, where~$L_i$ stands for the sublink of~$L$ made of the
components of color~$i$. Colored links should be thought of as natural generalizations of oriented links,
which correspond to the case~$m=1$.

The corresponding generalized notion of Seifert surface is due to Cooper~\cite{Coo82},
who defined it for~2-component~2-colored links, objects which were subsequently extended to arbitrary colored
links in~\cite{Cim04}.

\begin{definition}%C-cplx
\label{def:C-cplx}
A {\em C-complex\/} for an~$m$-colored link~$L=L_1\cup\dots\cup L_m$ is a union~$F=F_1\cup\dots\cup F_m$ of surfaces embedded in~$S^3$ such that~$F$ is connected and satisfies the following conditions:
\begin{enumerate}
    \item for all~$i$, the surface~$F_i$ is a Seifert surface for~$L_i$;
    \item for all~$i\neq j$, the surfaces~$F_i$ and~$F_j$ are either disjoint or intersect in a finite number of {\em clasps\/} (see Figure~\ref{clasp-regular});
    \item for all~$i,j,k$ pairwise distinct, the intersection~$F_i\cap F_j\cap F_k$ is empty.
\end{enumerate}
Such a C-complex is said to be {\em totally connected} if~$F_i\cap F_j$ is non-empty for all~$i\neq j$.
\end{definition}

The existence of a (totally connected) C-complex for any given colored link is fairly easy to establish, see~\cite{Cim04}. On the other hand, relating two C-complexes for isotopic colored links is more difficult~\cite{Cim04,C-F08}, and the correct version appeared only recently, see Theorem~1.3 of~\cite{DMO21}.

In the context of our work, we need to know how two totally connected C-complexes for isotopic colored links are related. In fact, Theorem~1.3 of~\cite{DMO21} pertains to a more general version of C-complexes: they are not necessarily totally connected, and the surfaces~$F_i$ are not assumed to be connected.
Fortunately, the results of~\cite{DMO21} do imply the following lemma.
In order not to break the flow of the reading, we differ its proof to Appendix~\ref{sec:appendix}.

\begin{figure}%regular crossing
\centering
\begin{overpic}[width=10cm]{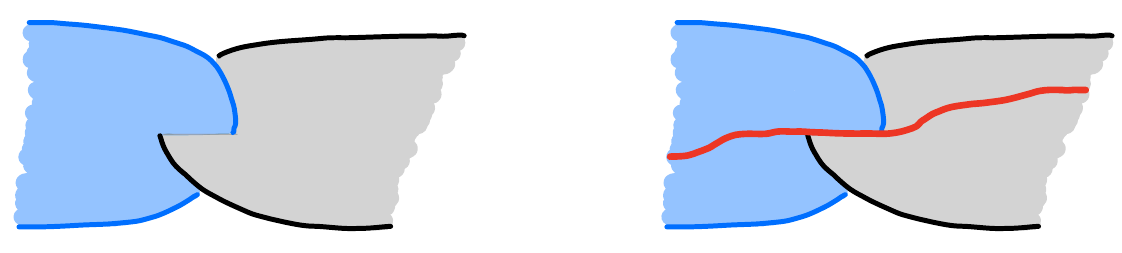}
 \put (5,15){$F_i$}
    \put (32,13){$F_j$}
    \end{overpic}
\caption{Left: a clasp intersection. Right: a well-behaved cycle crossing such a clasp intersection.}
\label{clasp-regular}
\end{figure}

\begin{lemma}[\cite{DMO21}]
Let $F$ and $F'$ be two totally connected C-complexes for isotopic colored links. Then, there exist totally connected C-complexes $F=F^1,F^2,\dots,F^{n-1},F^n=F'$ such that for all~$1\le k <n$,~$F^{k+1}$ is obtained from $F^k$ by one of the following moves or its inverse:
\begin{itemize}
\item[(T0)] ambient isotopy;
\item[(T1)] handle attachment along one of the surfaces (see Figure~\ref{ST1-2}, top);
\item[(T2)] add a ribbon intersection and push along an arc (see Figure \ref{T23}, left);
\item[(T3)] pass through a clasp (see Figure \ref{T23}, right);
\item[(T4)] replace a push along an arc by a push along another arc (see Figure~\ref{T4}).
\end{itemize}
\label{S-equivalence}
\end{lemma}

\begin{figure}[h]%T2
\centering
\begin{overpic}[width=6.5cm]{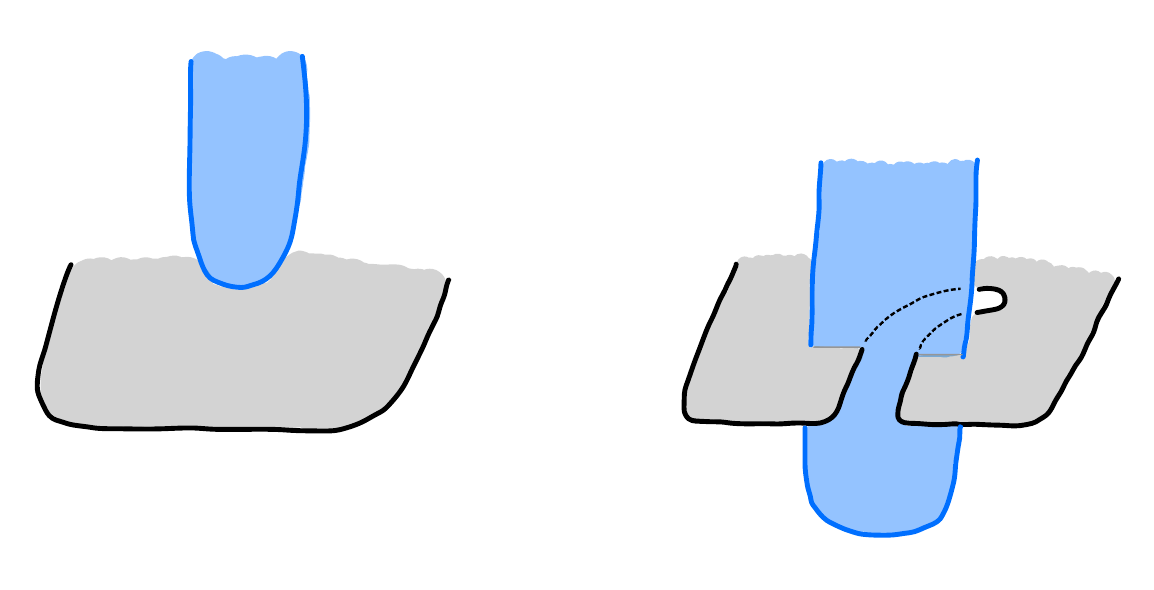}
\put(45,20){$\longrightarrow$}
\put(44,25){$\mathrm{(T2)}$}
\end{overpic}
\hskip1.5cm
\begin{overpic}[width=6.5cm]{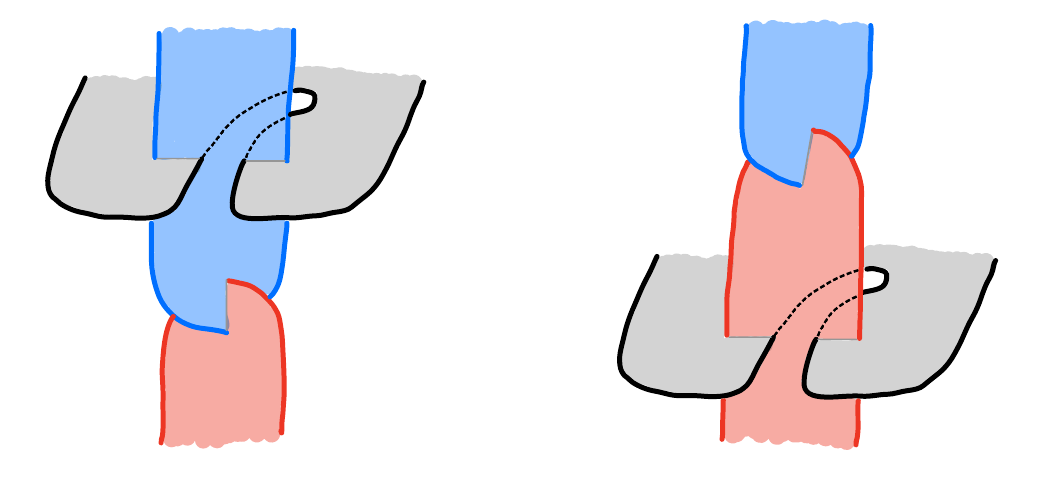}
\put(45,20){$\longrightarrow$}
\put(44,25){$\mathrm{(T3)}$}
\end{overpic}
\caption{The movements~(T2) and~(T3).}
\label{T23}
\end{figure}

\begin{figure}[h]
\centering
\begin{overpic}[width=10cm]{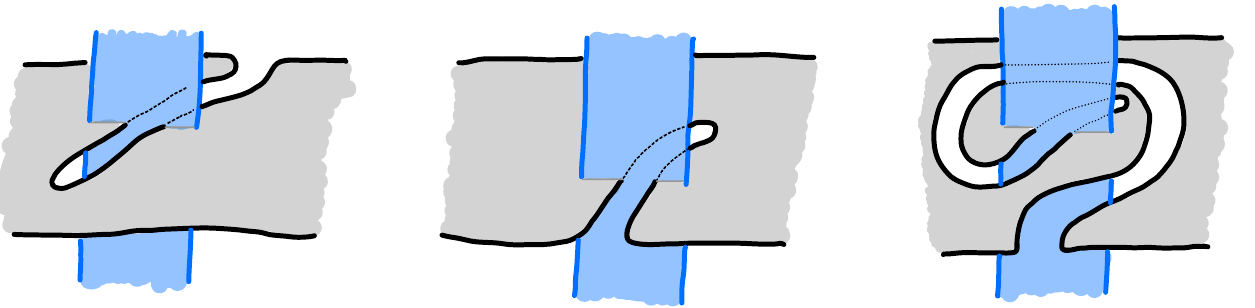}
\put(28,10){$\longleftarrow$}
\put(27,13.5){$\mathrm{(T4_a)}$}
\put(67,10){$\longrightarrow$}
\end{overpic}
\caption{Examples of~(T4) movements. }
\label{T4}
\end{figure}

\subsection{Generalized Seifert matrices}
\label{sub:GSM}

C-complexes allow to define {\em generalized Seifert forms\/} as follows.
For any sign~$\varepsilon=(\varepsilon_1,\dots,\varepsilon_m)\in\{\pm\}^m$, let
\[
\alpha^\varepsilon\colon H_1(F)\times H_1(F)\longrightarrow\mathbb{Z}
\]
be the bilinear form given by~$\alpha^\varepsilon(x,y)=\lk(x^\varepsilon,y)$,
where~$x^\varepsilon\subset S^3\setminus F$ denotes a cycle representating the homology class~$x\in H_1(F)$ pushed in the~$\varepsilon_i$-normal direction off~$F_i$ for all~$i=1,\dots,m$.
For this to make sense, one needs this cycle to be well-behaved when crossing clasps, as illustrated in
the right-hand part of Figure~\ref{clasp-regular}.
This is not an issue, as any homology class in~$H_1(F)$ can be represented by
such well-behaved cycles.
We denote by~$A^\varepsilon=A_F^\varepsilon$ the corresponding {\em generalized Seifert matrices\/}, defined with respect to a fix basis of~$H_1(F)$. Note the equality~$A^{-\varepsilon}=(A^\varepsilon)^{\T}$ for all~$\varepsilon\in\{\pm\}^m$.

Note that in the case~$m=1$, a C-complex for a~1-colored link~$L$ is nothing but a Seifert surface for
the oriented link~$L$, a matrix~$A^-$ corresponding to~$\varepsilon=-$ is a standard Seifert matrix,
and~$A^+$ its transpose.

\medskip

As recalled in Section~\ref{sub:Arf-L}, two Seifert matrices for isotopic oriented links are related by
a sequence of moves generating an equivalence relation known as S-equivalence.
Since two \dc{totally connected} C-complexes~$F,F'$ for isotopic colored links~$L,L'$ are related by the sequence of moves
of Lemma~\ref{S-equivalence}, the corresponding generalized Seifert matrices~$A^\varepsilon_{F},A^\varepsilon_{F'}$ are related by a sequence of transformations
(which in general depend on~$\varepsilon$).
\dc{Such computations were performed in~\cite{Cim04,DMO21} for general C-complexes, but we need to
consider the restrictive case of totally connected ones, yielding slightly different results.}

\dc{Once again, in order not to include technical statements in this background section, we differ
these computations to Appendix~\ref{sec:appendix}.}

\section{Arf invariants of colored links}
\label{sec:main}

The goal of this section is to explain how the classical Arf invariant defined in Section~\ref{sub:Arf-L} can be extended to colored links, i.e.\ to prove Theorem~\ref{thm:intro}. More precisely, we start in Section~\ref{sub:main}
with a discussion leading to the detailed statement of our main result. Its proof is given in Section~\ref{sub:ind} where we check independence of the C-complex.

\subsection{Main result}
\label{sub:main}

From now on, we work over the field~$\Z_2$ of modulo~2 integers, that will often be omitted from the notation.
For example, given a C-complex~$F$ for an~$m$-colored link~$L$ and a sign~$\varepsilon\in\{\pm\}^m$,
we still denote by~$\alpha^\varepsilon=\alpha_F^\varepsilon$ the form on~$H_1(F;\Z_2)$ given by the modulo~2 reduction of
the generalized Seifert form, and by~$A^\varepsilon=A_F^\varepsilon$ the modulo~2 reduction of a corresponding generalized
Seifert matrix.

To define a (generalized) Arf invariant of~$L$, we need a quadratic form.
For any~$\varepsilon\in\{\pm\}^m$, let us denote by~$q_F^\varepsilon$
(or by~$q^\varepsilon$ when no confusion is possible) the map
\[
q_F^\varepsilon\colon H_1(F;\Z_2)\to\Z_2
\]
defined by~$q_F^\varepsilon(x)=\alpha^\varepsilon_F(x,x)=\lk(x^\varepsilon,x)$.
As in the classical case, we have
\[
q_F^\varepsilon(x+y)+q_F^\varepsilon(x)+q_F^\varepsilon(y)=\lk(x^\varepsilon,y)+\lk(y^\varepsilon,x)=\lk(x^\varepsilon+x^{-\varepsilon},y)=:B_F^\varepsilon(x,y)
\]
for any~$x,y\in H_1(F;\mathbb{Z}_2)$. Since~$B_F^\varepsilon$ a bilinear form (with matrix~$A_F^\varepsilon+(A_F^\varepsilon)^{\T}$), we see that the map~$q_F^\varepsilon$ is a quadratic form. 
More generally, given any subset~$E$ of~$\{\pm\}^m$, the map
\[
q_F^E:=\sum_{\varepsilon\in E}q_F^\varepsilon\colon H_1(F;\Z_2)\to\Z_2
\]
is a quadratic form, with bilinear form~$B_F^E:=\sum_{\varepsilon\in E}B_F^\varepsilon$
admitting the matrix $A_F^E+(A_F^E)^{\T}$, with $A_F^E:=\sum_{\varepsilon\in E}{A_F^\varepsilon}$.

As explained in Section~\ref{sub:Arf-q}, there are now three possibilities for the quadratic form~$q_F^E$:
it might vanish on~$\rad(B_F^E)$, yielding the possible values~$\Arf(q_F^E)=0$ and~$\Arf(q_F^E)=1$;
or it might not vanish on~$\rad(B_F^E)$, yielding the third ``value'', namely~$\Arf(q_F^E)$ undefined.
This may or may not yield an invariant of~$L$, i.e.\ be independent of the choice of~$F$.
If this is the case, then one can set
\[
\Arf_{\!E}(L):=\Arf(q_F^E)\,,
\]
which might take the three values~$0,1$ or undefined.

For example, the empty set~$E=\emptyset$ yields the trivial (non-singular) quadratic form %$\overline{q_F^E}$
on the trivial space, leading to a well-defined, yet identically zero invariant.
On the other hand, for~$m=1$, both choices~$E=\{+\}$ and~$E=\{-\}$ lead to the classical Arf invariant~$\Arf_{\!E}(L)=\Arf(L)$, undefined unless~$\lk(K,L\setminus K)$ is even for each component~$K$ of~$L$, as explained in Section~\ref{sub:Arf-L}.

The remaining cases are covered by the following reformulation of Theorem~\ref{thm:intro} together with
the ensuing remark.

\begin{theorem}
\label{thm:main}
Let~$L$ be an~$m$-colored link. Given a totally connected C-complex~$F$ and a subset~$E$ of~$\{\pm\}^m$, let~$q_F^E\colon H_1(F;\Z_2)\to\Z_2$ be the associated quadratic form. If~$E$ is of even cardinality, then the value of $\Arf(q_F^E)$,~$0, 1$, or undefined, does not depend on the choice of~$F$.
\end{theorem}

\begin{remark}
\label{rem:Eodd}
If~$E$ has odd cardinality, then the conclusion of the theorem does not hold unless we are in the classical case of~$m=1$.

Let us first sketch why, if~$E$ consists of a singleton~$\{\varepsilon\}$,
any~$m$-colored link~$L$ with~$m>1$ admits two totally connected C-complexes~$F$ and~$F'$ leading to quadratic forms~$q^E_F$
and~$q^E_{F'}$ with different Arf invariants.
Indeed, let us consider an arbitrary C-complex~$F$  for~$L$, and let~$F(\varepsilon)$ denote the oriented
surface obtained from~$F$ by applying the transformation illustrated in Figure~\ref{transformation} to each clasp.
There is a canonical isomorphism~$H_1(F)\simeq H_1(F(\varepsilon))$ yielding a congruence
between the generalized Seifert form~$\alpha^\varepsilon$ on~$F$ and the classical Seifert
form~$\alpha^-$ on~$F(\varepsilon)$. As a consequence, we have
\[
\Arf(q_F^E)=\Arf(q_F^\varepsilon)=\Arf(q_{F(\varepsilon)})=\Arf(\partial F(\varepsilon))\,,
\]
the classical Arf invariant of the oriented link~$\partial F(\varepsilon)$.
Now, applying this transformation to two C-complexes~$F$ and~$F'$ related by a
move~(T2) yields two links~$\partial F(\varepsilon)$ and~$\partial F'(\varepsilon)$
which are related by a band move and the addition of an unknot,
see Figure~\ref{odd}.
This process enables us to produce two totally connected C-complexes for~$L$ having Arf invariants with different values.

\begin{figure}%transformation into surface
\centering
\begin{overpic}[width=7cm]{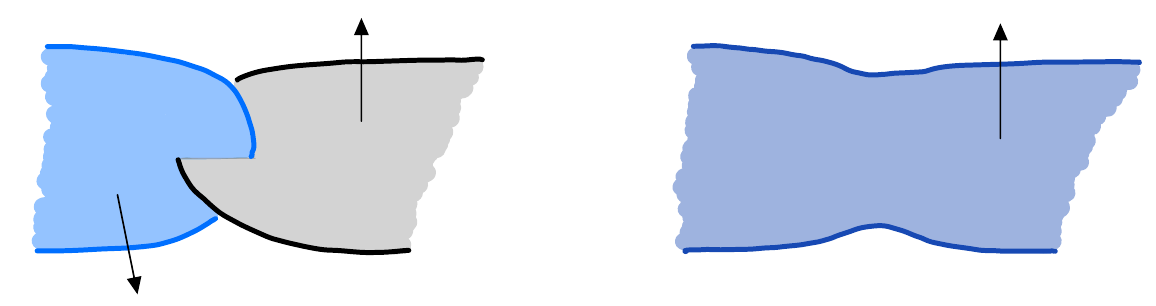}
\put(10,15){$F_i$}
\put(29,10){$F_j$}
\put(70,13){$F(\varepsilon)$}
\put(30,27){$\varepsilon_j$}
\put(12,-3){$\varepsilon_i$}
\put(48,14){$\longrightarrow$}
\end{overpic}
\caption{Transforming a C-complex~$F$ into an oriented surface~$F(\varepsilon)$.}
\label{transformation}
\end{figure}

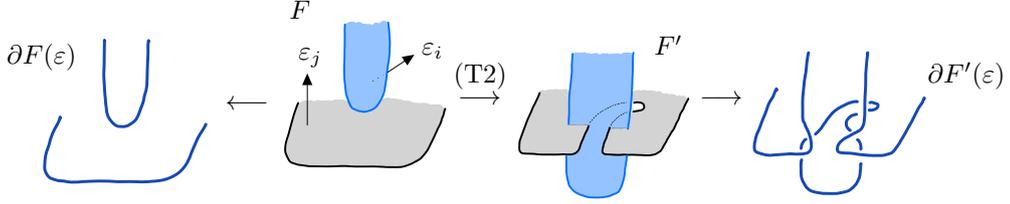
\begin{figure}%undefinition
\centering
\begin{overpic}[width=0.8\linewidth]{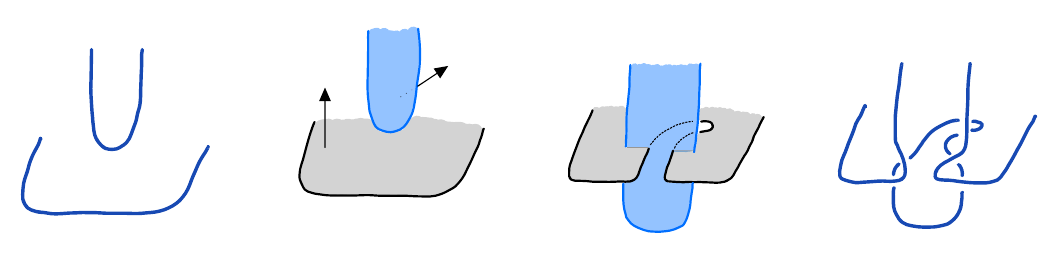}
\put(69,19){$F'$}
\put(29,23){$F$}
\put(-2,18){$\partial F(\varepsilon)$}
\put(99,16){$\partial F'(\varepsilon)$}
\put(30,18.5){$\varepsilon_j$}
\put(43.5,19){$\varepsilon_i$}
\put(74,13.5){$\longrightarrow$}
\put(47.5,13.5){$\longrightarrow$}
\put(22,13){$\longleftarrow$}
\put(47,16){$\mathrm{(T2)}$}
\end{overpic}
\caption{The links~$\partial F(\varepsilon)$ and~$\partial F'(\varepsilon)$
obtained via~$F$ and~$F'$ related by a move~(T2).}
\label{odd}
\end{figure}

The general case of~$E$ with odd cardinality is slightly more technical.
As it turns out, one can show that any totally connected C-complex~$F_0$ such that~$\Arf(q_{F_0}^E)$ is defined
produces a link admitting two totally connected C-complexes~$F$ and~$F'$ with exactly one of $\Arf(q^E_F)$ and~$\Arf(q^E_{F'})$ defined. However, the details are rather cumbersome and will not be included in this note.
\end{remark}

\begin{remark}
\label{rem:c-ex}
In general, Theorem~\ref{thm:main} does not hold either without the assumption of total connectivity.
For example, consider the two C-complexes~$F$ and~$F'$ illustrated in Figure~\ref{c-ex} (left): they
are related by a move~(T3), recall Figure~\ref{ST3}, and in particular yield isotopic colored links.
However, the C-complex~$F'$ is not totally connected.
Assuming that the components are ordered blue-black-red and that the left of Figure~\ref{c-ex}
shows the positive sides of the three surfaces, the choice~$E=\{+++,++-\}$ leads to a first quadratic form~$q^E_F$ which vanishes on~$\rad(B^E_F)$, while the second quadratic form~$q^E_{F'}$ does not vanish on~$\rad(B^E_{F'})$.
Hence, the value of~$\Arf(q_F^E)$ in general does depend on the choice of the C-complex~$F$ if is allowed not to be totally connected.
\end{remark}

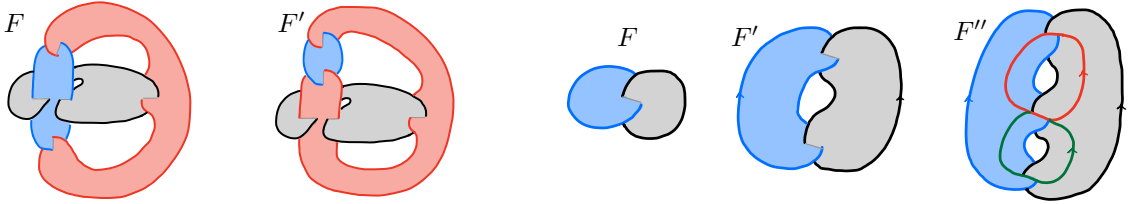
\begin{figure}%c-ex T3
\centering
\begin{overpic}[width=0.4\linewidth]{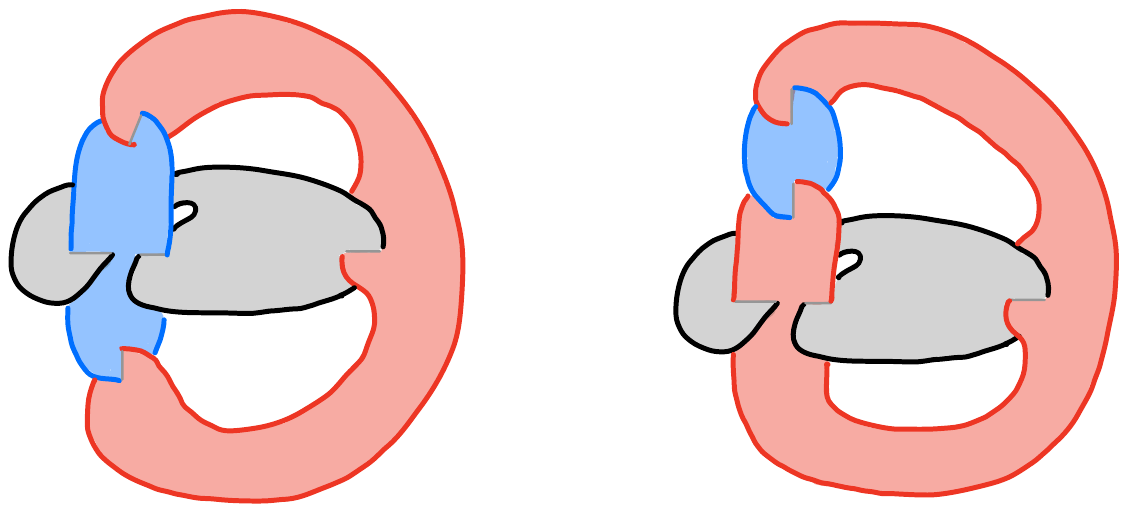}
\put(60,38){$F'$}
\put(0,38){$F$}
\end{overpic}
\hskip1.2cm
\begin{overpic}[width=0.5\linewidth]{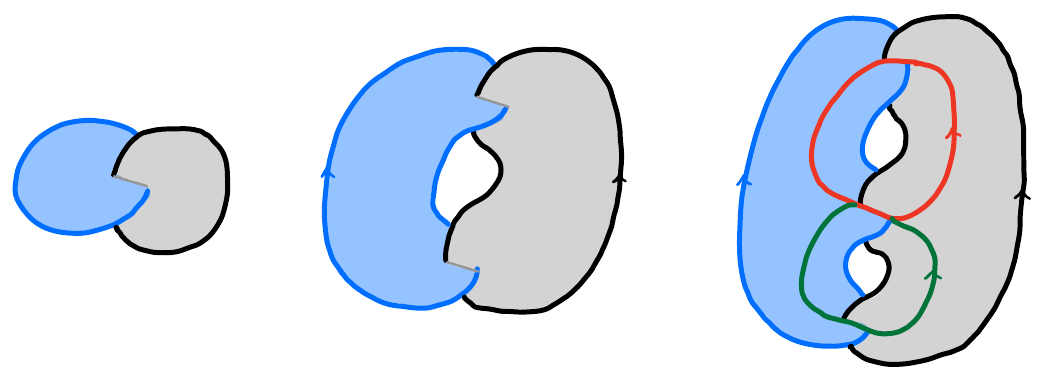}
\put(30,28){$F'$}
\put(69,29){$F''$}
\put(10,28){$F$}
\end{overpic}
\caption{Left: the counterexample of Remark~\ref{rem:c-ex}. Right: the links and C-complexes of Example~\ref{ex:Arf}}
\label{c-ex}
\end{figure}

We now present a couple of easy examples.

\begin{example}
\label{ex:Arf}
Let us focus on the case of~$2$ colors. Because of the equality~$A^{-\varepsilon}=(A^\varepsilon)^{\T}$, a
choice of~$E$ in~$\{\{++,--\},\{+-,-+\},\{++,+-,-+,--\}\}$ yields a trivial invariant $\Arf_{\!E}\equiv 0$,
and the same obviously holds for~$E=\emptyset$. For the same reason, we have~$\Arf_{\!E}\equiv\Arf_{\!-E}$,
where~$-E$ denotes the set~$E$ with all signs changed. Hence, we are left with~$E=\{++,+-\}$ and~$\{++,-+\}$
yielding potentially non-trivial invariants.

First, consider the~2-colored link~$L$ given by the (positive or negative) Hopf link. As illustrated in the right of Figure~\ref{c-ex}, it admits a contractible C-complex~$F$; therefore, we get~$\Arf_{\!E}(L)=0$.

Then, let~$L'$ denote the~2-colored link given by the~$(2,4)$-torus link. The C-complex~$F'$
illustrated in the right of Figure~\ref{c-ex} leads to~$A^{++}=A^{--}=(-1)$ and~$A^{+-}=A^{-+}=(0)$. Hence, the signs~$E=\{++,+-\}$ and~$\{++,-+\}$ give the
matrices~$A^E=(1)$ and~$B^E=(0)$, so we get~$\Arf_{\!E}(L')$ undefined.

Finally, let~$L''$ denote the~$(2,6)$-torus link.
The C-complex~$F''$ together with the basis of~$H_1(F'')$ illustrated in the right of Figure~\ref{c-ex} yield the generalized Seifert matrices~$A^{++}=(A^{--})^{\T}=\left(\begin{smallmatrix}-1&0\\ -1&-1\end{smallmatrix}\right)$ and~$A^{+-}=A^{-+}=\left(\begin{smallmatrix}0&0\\ 0&0\end{smallmatrix}\right)$.
Both choices~$E=\{++,+-\}$ and~$\{++,-+\}$ lead to the
matrices~$A^E=\left(\begin{smallmatrix}1&0\\ 1&1\end{smallmatrix}\right)$ and~$B^E=\left(\begin{smallmatrix}0&1\\ 1&0\end{smallmatrix}\right)$, so we get~$\Arf_{\!E}(L'')=1$, for example using Equation~\eqref{formularf}.
\end{example}

\subsection{Proof of Theorem \ref{thm:main}}
\label{sub:ind}

Let us fix an~$m$-colored link~$L$ with~$m>1$ and a subset~$E\subset\{\pm\}^m$ of even cardinality. 
We want to show that if~$F$ and~$F'$ are two totally connected C-complexes for~$L$, then the values of~$\Arf(q^E_F)$ and of~$\Arf(q^E_{F'})$ coincide.

\dc{The most natural and direct way to show this result is to use the four transformations involved in the generalized S-equivalence of Appendix~\ref{sec:appendix}.
Instead, we will make use of the following lemma which introduces a new move~(T5):
this allows to reduce the number of transformations from four to three, thus yielding a slightly shorter proof
of Theorem~\ref{thm:main}. Note however that this new move might change the isotopy type of the underlying colored link, and can therefore not be used for generalized S-equivalence.}

\begin{lemma}%T5 invariance
Let $F$ and $F'$ be two totally connected C-complexes and $E\subset\{\pm\}^m$ of even cardinality. If $F$ and $F'$ are related by the move~(T5) described in Figure~\ref{T5}, then the values of~$\Arf(q_F^E)$ and of~$\Arf(q^E_{F'})$ coincide.
\label{T5invariance}
\end{lemma}

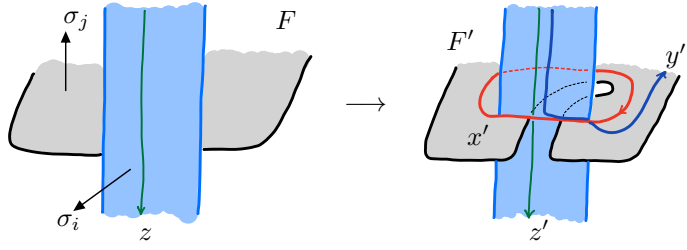
\begin{figure}%T5
\centering
\begin{overpic}[width=0.6\linewidth]{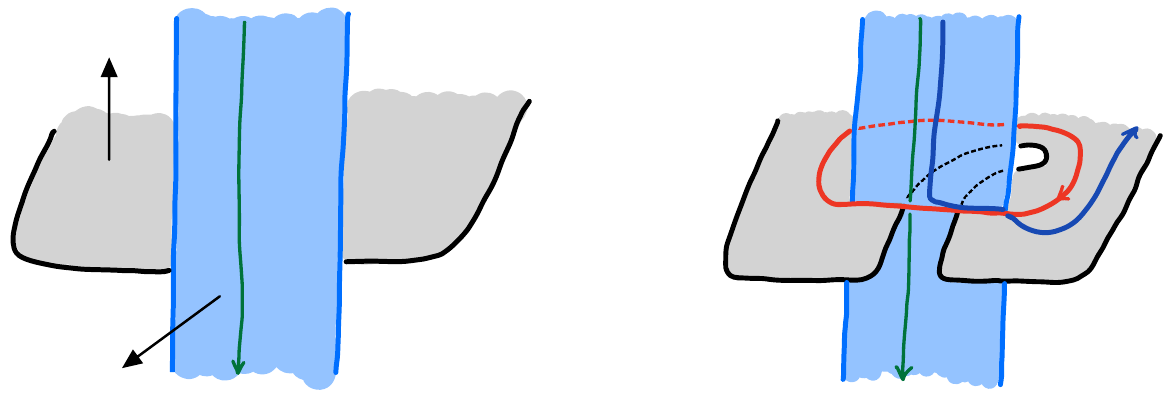}
\put(65,27){$F'$}
\put(40,30){$F$}
\put(8,1){$\sigma_i$}
\put(9,31){$\sigma_j$}
\put(20,-1){$z$}
\put(77,-1){$z'$}
\put(68,13){$x'$}
\put(97,25){$y'$}
\put(50,18){$\longrightarrow$}
\end{overpic}
\caption{The move~(T5).}
\label{T5}
\end{figure}

We will first introduce some notation, then prove Theorem~\ref{thm:main} assuming Lemma \ref{T5invariance},
and finally check this lemma.

\begin{notation}
\label{not:n}
Given~$E\subset\{\pm\}^m$ and a fixed color~$i$ and sign~$\sigma$, we write~$n_i^\sigma$ for the cardinality of~$\{\varepsilon\in E\,|\,\varepsilon_i=\sigma\}$.
Similarly, given signs~$\sigma,\tau$ and colors~$i,j$, we
denote by~$n_{ij}^{\sigma\tau}$ the cardinality of~$\{\varepsilon\in E\,|\,\varepsilon_i=\sigma,\,\varepsilon_j=\tau\}$.
\end{notation}

\begin{proof}[Proof of Theorem~\ref{thm:main}.]
By Lemma~\ref{S-equivalence}, we only need to show that if a totally connected C-complex~$F'$ is obtained from another totally connected C-complex~$F$ by one of the movements (T0)--(T4), then the values~$\Arf(q^E_{F})$ and~$\Arf(q^E_{F'})$ coincide.

Since~(T0) is an ambient isotopy, the equality is automatically satisfied. The moves~(T2) and~(T3) being compositions of~(T0) and~(T5), Lemma \ref{T5invariance} implies that they do not change the value of the Arf invariant. 

Let us now suppose that, for some~$1\le i\le m$ and~$\sigma_i\in\{\pm\}$, the C-complex~$F'$ is obtained from~$F$ by a move~(T1) on the~$\sigma_i$-side of~$F_i$.
By Equation~\eqref{eq:T1}, the sums of generalized Seifert matrices~$A^E_F=\sum_{\varepsilon\in E} A^\varepsilon_F$ and~$A^E_{F'}=\sum_{\varepsilon\in E} A^\varepsilon_{F'}$ are related by
\[
A_{F'}^E=\begin{pmatrix}
A^E_F & 0 & n_i^{\sigma_i}\xi(\sigma_i)+n_i^{-\sigma_i}\xi(-\sigma_i)\\
0 & 0 & n_i^{\sigma_i}\\
n_i^{\sigma_i}\xi(-\sigma_i)^{\T}+n_i^{-\sigma_i}\xi(\sigma_i)^{\T} & n_i^{-\sigma_i} & 0
\end{pmatrix}\,,
\]
assuming Notation~\ref{not:n}.
Since $E$ is of even cardinality, the integers~$n_i^{\sigma_i}$ and~$n_i^{-\sigma_i}$ have the same parity, so the associated bilinear forms satisfy~$B_{F'}^E = B_F^E \oplus\left(\begin{smallmatrix}0&0\\ 0&0\end{smallmatrix}\right)$.
In other words, we have~$\rad(B_{F'}^E)=\rad(B_F^E)\oplus\Z_2 x\oplus\Z_2 y$, with~$x,y$ corresponding to
the final two rows and columns in~$A_{F'}^E$. Since~$q^E_{F'}(x)=q^E_{F'}(y)=0$, the Arf invariants of~$q^E_{F'}$ and of~$q^E_{F}$ have the same value.

Finally, the move~(T4) can be obtained as follows: first use an ambient isotopy (twisting the vertical band)
to get the arc into one of the two situations illustrated in the left and center of Figure~\ref{T4}, then (if needed) use the inverse of~$\mathrm{(T4_a)}$ to get to the center of this figure,
then move~(T5), then an ambient isotopy to untwist the vertical band,
and finally the inverse of~(T5) to get the arc into standard position.
Thus, the invariance under a general move~(T4) can be reduced to the invariance under the specific move~$\mathrm{(T4_a)}$. If~$F'$ is obtained from~$F$ via such a move, then
Equations~\eqref{eq:T4} and~\eqref{eq:T4'} imply that the generalised Seifert matrices~$A^\varepsilon_F$ and~$A^\varepsilon_{F'}$ have identical diagonals (recall that we have~$n=0$), and satisfy
\[
B^\varepsilon_F=A^\varepsilon_F+(A^{\varepsilon}_F)^{\T}=A^\varepsilon_{F'}+(A^\varepsilon_{F'})^{\T}=B^\varepsilon_{F'}
\]
for all~$\varepsilon\in\{\pm\}^m$.
Thus, any combination of these matrices leads to the same Arf invariant.
\end{proof}

This concludes the proof of Theorem~\ref{thm:main}, assuming Lemma~\ref{T5invariance} whose demonstration
we now present.

\begin{proof}[Proof of Lemma \ref{T5invariance}]
Let us fix two totally connected C-complexes~$F$ and~$F'$ that are related by the move~(T5) illustrated in Figure~\ref{T5}, whose notation we assume.
Since~$F$ is totally connected, we have~$H_1(F')=H_1(F)\oplus \Z x'\oplus \Z y'$, with~$x',y'\subset F_i\cup F_j$ the cycles represented in the right side of Figure~\ref{T5}.
Up to a move~(T1) along~$F_i$ if needed, a basis of~$F$ can be chosen to contain a unique cycle~$z$ as in Figure~\ref{T5}.
Consider the symbols~$\delta$ and~$\chi$ introduced in Notation~\ref{notation}.
In the appropriate bases, the generalized Seifert matrix~$A^\varepsilon_{F'}$ is obtained from~$A^\varepsilon_F$ via
\[
A^\varepsilon_{F'}=\begin{pmatrix}
A^\varepsilon_F & \delta(\varepsilon_i,-\sigma_i)\chi(z) & \xi(\varepsilon_i,\varepsilon_j)\\
\delta(\varepsilon_i,\sigma_i)\chi(z)^{\T} & 0 & \delta(\varepsilon_i,\sigma_i)\delta(\varepsilon_j,\sigma_j)\\
\xi(-\varepsilon_i,-\varepsilon_j)^{\T} & \delta(\varepsilon_i,-\sigma_i)\delta(\varepsilon_j,-\sigma_j)  & \ell(\varepsilon_i,\varepsilon_j) 
\end{pmatrix}\,
\]
with~$\ell(-\varepsilon_i,-\varepsilon_j)=\ell(\varepsilon_i,\varepsilon_j)\in\Z$ and~$\xi(\varepsilon_i,\varepsilon_j)$ some integral vector only depending on~$\varepsilon_i$ and~$\varepsilon_j$.

Assuming Notation~\ref{not:n} and setting
\[
n_\text{e}:=n_{ij}^{\sigma_i\sigma_j}+n_{ij}^{-\sigma_i-\sigma_j}\,,\quad n_\text{o}:=n_{ij}^{\sigma_i-\sigma_j}+n_{ij}^{-\sigma_i\sigma_j}\,,
\]
this leads to
\begin{equation}
\label{eq:A-A'}
A^E_{F'}=\sum_{\varepsilon\in E}A^{\varepsilon}_{F'}=\begin{pmatrix}
A^E_F&n_i^{-\sigma_i}\chi(z)& \sum_{\varepsilon\in E}\xi(\varepsilon_i,\varepsilon_j)\\
n_i^{\sigma_i}\chi(z)^{\T}&0&n_{ij}^{\sigma_i\sigma_j}\\
\sum_{\varepsilon\in E}\xi(-\varepsilon_i,-\varepsilon_j)^{\T}&n_{ij}^{-\sigma_i-\sigma_j}& n_\text{e}\ell(\sigma_i,\sigma_j)+n_\text{o}\ell(\sigma_i,-\sigma_j)\end{pmatrix}\,.
 \end{equation}
By assumption, the integer~$n_i^{\sigma_i}+n_i^{-\sigma_i}=\vert E\vert$ is even, yielding
\[
B^E_{F'}=A^E_{F'}+(A^E_{F'})^{\T}=\begin{pmatrix}
B^E_F&0&\xi^E\\
0&0&n_\text{e}\\
(\xi^E)^{\T}&n_\text{e}&0\end{pmatrix},
\]
where
\[
\xi^E= \sum_{\varepsilon\in E}\xi(\varepsilon_i,\varepsilon_j)+ \sum_{\varepsilon\in E}\xi(-\varepsilon_i,-\varepsilon_j)=
n_\text{e}(\xi(\sigma_i,\sigma_j)+\xi(-\sigma_i,-\sigma_j))
+n_\text{o}(\xi(\sigma_i,-\sigma_j)+\xi(-\sigma_i,\sigma_j))\,.
\]
Since
\[
n_\text{e}+n_\text{o}=n_{ij}^{\sigma_i\sigma_j}+n_{ij}^{-\sigma_i-\sigma_j}+n_{ij}^{\sigma_i-\sigma_j}+n_{ij}^{-\sigma_i\sigma_j}=\vert E \vert
\]
is even, the two integers~$n_\text{e}$ and~$n_\text{o}$ have the same parity.
If it is even, then by the computations above, we have~$B^E_{F'}=B^E_F\oplus\left(\begin{smallmatrix}0&0\\ 0&0\end{smallmatrix}\right)$ while the matrix~$A^E_{F'}$ is of the form
\[
A^E_{F'}=\begin{pmatrix} A^E_F& \ast&\ast\\ \ast&0&\ast\\ \ast&\ast&0\end{pmatrix}\,.
\]
The values of~$\Arf(q^E_{F'})$ and~$\Arf(q^E_F)$ are then easily seen to coincide, as
in the proof of the invariance under move~(T1) above.
If~$n_\text{e}$ and~$n_\text{o}$ are odd, then the matrix~$B^E_{F'}$ is of the form
\begin{equation}
\label{eq:B-B'}
B^E_{F'}=\begin{pmatrix} B^E_F& 0&\ast\\ 0&0&1\\ \ast&1&0\end{pmatrix}\sim\begin{pmatrix} B^E_F& 0&0\\ 0&0&1\\0&1&0\end{pmatrix}\,,
\end{equation}
with~$\sim$ standing for congruence, leading to~$\rad(B^E_{F'})=\rad(B^E_F)\subset H_1(F;\Z_2)$.
By~\eqref{eq:A-A'}, the restriction of~$q^E_{F'}$ to~$H_1(F;\Z_2)$ coincides with~$q^E_{F}$,
so~$q^E_{F'}$ vanishes on~$\rad(B^E_{F'})$ if and only if~$q^E_{F}$ vanishes on~$\rad(B^E_{F})$.
Let us assume that this is the case, so that neither~$\Arf(q^E_{F'})$ nor~$\Arf(q^E_{F})$ is undefined.
By Equation~\eqref{eq:B-B'}, if~$(a_k)_k$ is a basis of~$H_1(F;\Z_2)/\rad(B^E_{F})$ which is symplectic with respect to~$B^E_F$, then a basis of~$H_1(F';\Z_2)/\rad(B^E_{F'})$ which is symplectic with respect to~$B^E_{F'}$ is given
by~$(a'_k)_k\cup(x,y)$, with $a'_k\in\{a_k, a_k+x\}$. By~\eqref{eq:A-A'},
we have~$q^E_{F'}(x)=0$ and~$q^E_{F'}(a_k')=q^E_F(a_k)$ for all~$k$. Equation~\eqref{formularf} now implies that~$\Arf(q^E_{F'})$ and~$\Arf(q^E_{F})$ coincide.
\end{proof}

\section{Determination by linking numbers}
\label{sec:lk}

The aim of this last section is to give the proof of Proposition~\ref{prop:intro}, whose statement we recall
for the reader's convenience.

\begin{proposition}
\label{prop:lk}
For any~$m$-colored link~$L$ and any~$E\subset\{\pm\}^m$  of even cardinality,
the value of~$\Arf_{\!E}(L)$ is determined by the linking numbers of the components of~$L$.
\end{proposition}

Our proof of this result relies on four lemmas.
The first one is best formulated using the following terminology:
two colored links are said to be {\em homotopic\/} if they can be
related by isotopies (preserving the orientation and color of each component) and crossing changes between strands of the same color.
This first lemma states that generalized Arf invariants are invariant
under homotopy of colored links.

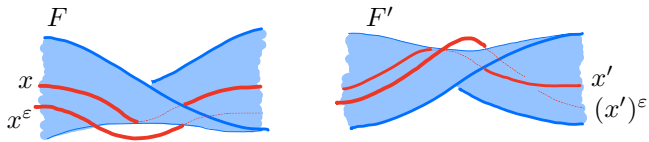
\begin{figure}%crossing induced by a twist on a surface
\centering
\begin{overpic}[width=0.5\linewidth]{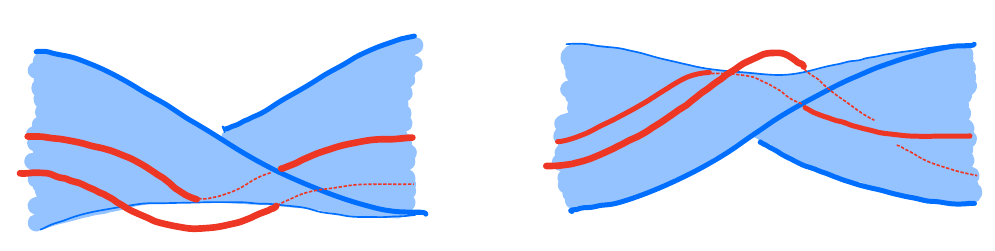}
\put(60,22){$F'$}
\put(4,22){$F$}
\put(-3,4){$x^{\varepsilon}$}
\put(99.5,11.5){$x'$}
\put(100,6){$(x')^\varepsilon$}
\put(-1,11){$x$}
\end{overpic}
\caption{A crossing change between two strands of the same color is realised by twisting a band.}
\label{twist}
\end{figure}

\begin{lemma}%homotopy
\label{lemma:homotopy1}
If~$L$ and~$L'$ are homotopic~$m$-colored links,
then for any~$E\subset\{\pm\}^m$ of even cardinality, we have~$\Arf_{\!E}(L)=\Arf_{\!E}(L')$.
\end{lemma}
\begin{proof}
Any crossing change between two strands of the same color can be realised by adding a full twist to a band in a C-complex, as illustrated in Figure~\ref{twist}.
Writing~$F$ and~$F'$ for the corresponding C-complexes,
one immediately sees that the groups~$H_1(F;\Z_2)$ and~$H_1(F';\Z_2)$ are canonically isomorphic;
let us denote by~$x'\in H_1(F';\Z_2)$ the image under this isomorphism of an arbitrary cycle~$x$ in~$F$.
We will  now check that~$q^E_{F}(x)=q^E_{F'}(x')$ for any~$x\in H_1(F;\Z_2)$, which implies the statement of the lemma.
Consider an arbitrary class in~$H_1(F;\Z_2)$. Since we are working modulo~2, this class is represented
by a cycle~$x\subset F$ which is either supported outside the part of~$F$ depicted in Figure~\ref{twist},
or crosses this band once. In the first case, we clearly have~$q^\varepsilon_{F}(x)=q^\varepsilon_{F'}(x')$
for all~$\varepsilon$ and the equality~$q^E_{F}(x)=q^E_{F'}(x')$ holds trivially.
In the second, and for any choice of signs~$\varepsilon\in\{\pm\}^m$, we have
the modulo~2 equality
\[
q_F^\varepsilon(x)=\lk(x^\varepsilon,x)=\lk((x')^\varepsilon,x')+1=q_{F'}^\varepsilon(x')+1\,,
\]
see Figure~\ref{twist}.
Since this holds for all~$\varepsilon$ and~$E$ is even, the equality~$q^E_{F}(x)=q^E_{F'}(x')$ follows.
\end{proof}

Before stating the next lemma, recall that~$n_{ij}^{\sigma\tau}$ denotes the cardinality of~$\{\varepsilon\in E\,|\,\varepsilon_i=\sigma,\,\varepsilon_j=\tau\}$.

\begin{lemma}%i,j-homotopy
Let $E\subset\{\pm\}^m$ be a set of even cardinality and $i,j$ be colors such that~$n_{ij}^{++}$
and~$n_{ij}^{--}$ have the same parity.
Then, given any colored links~$L,L'$ related by a crossing change between components of colors~$i$ and~$j$,
we have~$\Arf_{\!E}(L)=\Arf_{\!E}(L')$.
\label{lemma:homotopy}
\end{lemma}

\begin{proof}
First note that if the colors~$i$ and~$j$ coincide (a case that was not ruled out explicitly),
then we have~$n_{ii}^{++}+n_{ii}^{--}=\vert E\vert$ even, so the assumption is always verified,
and the conclusion holds by Lemma~\ref{lemma:homotopy1}. Therefore, we can now assume that~$i$ and~$j$ are distinct colors.
Any crossing change between two strands of different colors can be realised by adding a clasp in a C-complex, as illustrated in Figure~\ref{homotopy_colors}. Assuming the notations~$F,F'$ of this figure,
and using the fact that~$F$ is connected, we get a natural isomorphism~$H_1(F';\Z_2)\simeq H_1(F;\Z_2)\oplus \Z_2 y$, with~$y$ a cycle passing through the newly created clasp. Moreover, since~$F$ is totally connected,
one can assume that~$y$ is contained in~$F_i\cup F_j$.
Expressed in coherent bases, the generalized Seifert matrices are related by 
\[
A_{F'}^\varepsilon=\begin{pmatrix}
A^\varepsilon_{F} & \xi(\varepsilon_i,\varepsilon_j)\\  \xi(-\varepsilon_i,-\varepsilon_j)^{\T}&  \ell(\varepsilon_i,\varepsilon_j)\end{pmatrix}\,,
\]
with~$\ell(\varepsilon_i,\varepsilon_j)=\ell(-\varepsilon_i,-\varepsilon_j)\in\Z$ and~$\xi(\varepsilon_i,\varepsilon_j)$ some integral vector which only depends on~$\varepsilon_i$ and~$\varepsilon_j$.
By assumption, the two integers~$n_{ij}^{++}$ and~$n_{ij}^{--}$ have the same parity. Since
\[
n_{ij}^{++}+n_{ij}^{--}+n_{ij}^{+-}+n_{ij}^{-+}=\vert E\vert
\]
is even, the two integers~$n_{ij}^{+-}$ and~$n_{ij}^{-+}$ also have identical parity. Since~$\ell(\varepsilon_i,\varepsilon_j)=\ell(-\varepsilon_i,-\varepsilon_j)$, this leads to
\[
q^E_{F'}(y)=\sum_{\varepsilon\in E}\ell(\varepsilon_i,\varepsilon_j)=(n_{ij}^{++}+n_{ij}^{--})\ell(+,+)+(n_{ij}^{+-}+n_{ij}^{-+})\ell(+,-)=0\,.
\]
This also leads to
\[
\sum_{\varepsilon\in E}\xi(\varepsilon_i,\varepsilon_j)+\sum_{\varepsilon\in E}\xi(-\varepsilon_i,-\varepsilon_j)
=\sum_{\sigma,\tau=\pm}(n_{ij}^{\sigma\tau}+n_{ij}^{-\sigma-\tau})\xi(\sigma,\tau)=0\,,
\]
and thus to the identity~$B^E_{F'}=B^E_{F}\oplus(0)$. The equality~$\Arf(q^E_{F'})=\Arf(q^E_{F})$ follows.
\end{proof}

\begin{figure}%homotopy colors
\centering
\begin{overpic}[width=0.8\linewidth]{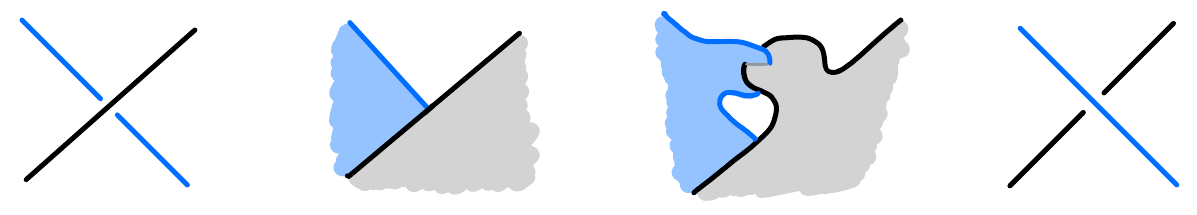}
\put(35,13){$F$}
\put(63,16.5){$F'$}
\put(48,9){$\longrightarrow$}
\end{overpic}
\caption{A crossing change between two strands of different colors can be realised by adding a clasp.}
\label{homotopy_colors}
\end{figure}

The following lemma shows that we can partition the colors in two sets, inside of which changing crossings will not affect the Arf invariant.

\begin{lemma}
Let~$E\subset\{\pm\}^m$ be of even cardinality. Then there exists a partition~$I_0\sqcup I_1$ of~$\{1,\dots, m\}$ such that for any~$m$-colored link~$L$, the value of~$\Arf_{\!E}(L)$ is invariant under crossing changes between components of colors belonging to the same subset of this partition.
\label{partition}
\end{lemma}

\begin{proof}
Fix an arbitrary color~$i\in\{1,\dots,m\}$. For~$s=0,1$, set
\[
I_s:=\{j\in\{1,\dots,m\}\,|\,n_{ij}^{++}+n_{ij}^{--}\equiv s \!\!\!\pmod 2 \}\,.
\]
This obviously defines a partition~$I_0\sqcup I_1$ of $\{1,\dots,m\}$, with~$i\in I_0$.
By Lemma~\ref{lemma:homotopy1} and Lemma~\ref{lemma:homotopy}, it only remains to check that for any~$j,k$
in the same subset of this partition, the integers~$n_{jk}^{++}$ and~$n_{jk}^{--}$ have the same parity.
To show this claim, let us first denote by~$n_{ijk}^{\sigma\tau\eta}$ the cardinality of~$\{\varepsilon\in E\,|\,\varepsilon_i=\sigma, \varepsilon_j=\tau, \varepsilon_k=\eta\}$.
Using the identity~$n_{ijk}^{+\sigma\tau}+n_{ijk}^{-\sigma\tau}=n_{jk}^{\sigma\tau}$ and variations thereof,
together with the fact that~$\vert E\vert$ is even, we obtain the following equalities modulo~2:
\begin{align*}
n_{jk}^{++}+n_{jk}^{--}&=n_{ijk}^{+++}+n_{ijk}^{-++}+n_{ijk}^{+--}+n_{ijk}^{---}=
n_{ij}^{+-}+n_{ij}^{-+}+n_{ik}^{++}+n_{ik}^{--}\\
	&=n_{ij}^{++}+n_{ij}^{--}+n_{ik}^{++}+n_{ik}^{--}\,.
\end{align*}
Since~$j$ and~$k$ belong to the same subset of the partition, this number is even, proving the claim.
\end{proof}

The final lemma, which was communicated to us by Christopher Davis, states that~$2$-colored links are classified up to homotopy by the linking numbers of their components of different colors.
More precisely, let us consider a~$2$-colored link~$L=L_1\cup L_2$. Given any
ordering of the components of~$L_1$ and any ordering of the components of~$L_2$, there is an associated
matrix~$\left(\lk(K_1,K_2)\right)_{K_1\subset L_1,K_2\subset L_2}$.
This matrix is not canonically associated to~$L$, but any two such matrices for~$L$
coincide up to reordering the rows and columns.

\begin{lemma}
Two~2-colored links~$L=L_1\cup L_2$ and~$L'=L_1'\cup L_2'$ are homotopic if and only if the associated
matrices~$\left(\lk(K_1,K_2)\right)_{K_1\subset L_1,K_2\subset L_2}$ and~$\left(\lk(K'_1,K'_2)\right)_{K'_1\subset L'_1,K'_2\subset L'_2}$ coincide up to reordering the rows and columns.
\label{lemma:classification}
\end{lemma}

\begin{proof}
The matrix~$\left(\lk(K_1,K_2)\right)_{K_1\subset L_1,K_2\subset L_2}$, considered up to reordering the rows and columns, is obviously invariant under homotopy.
It remains to prove that this matrix determines the homotopy class of a~2-colored link.
To do so, let us fix a~2-colored link~$L=L_1\cup L_2$ together with arbitrary orderings of the components of~$L_1$ and of the components of~$L_2$. We will show that~$L$ is homotopic to a link in a normal form
which is uniquely determined up to isotopy by the corresponding matrix~$\left(\lk(K_1,K_2)\right)_{K_1\subset L_1,K_2\subset L_2}$.

Using homotopy between strands of color~1 (blue), one can first turn~$L_1$ into the trivial link.
To make it more concrete, let us say that~$L_1$ is the closure of the trivial oriented braid
in the cylinder~$D^2\times[0,1]$, which we picture horizontally with the braid oriented left to right.
Up to isotopy, the components of~$L_1$ can be permuted arbitrarily, so we can assume that their ordering (say,
top to bottom) coincides with the fixed ordering.
As a second step, we can now use homotopy between strands of color~2 (red) so that each of the components of~$L_2$
lies within a different slice of the cylinder, as illustrated in Figure~\ref{Normalform} (left).
Using homotopy, the components of~$L_2$ can be permuted so that their ordering (left to right) coincides with the fixed ordering.

It now remains to put any fixed component~$K_2$ of~$L_2$ in a normal form (within the corresponding cylinder slice) determined by the linking numbers~$\{\lk(K_1,K_2)\}_{K_1\subset L_1}$.
To do so, first use homotopy within~$K_2$ to unknot it, and to put it in an ``increasing'' form inside its cylinder
slice as illustrated in Figure~\ref{Normalform} (center).
Finally, consider the transformation illustrated below, which can be realised via a homotopy between components
of color~1.

\begin{figure}[h]%"the movement for the normal form"
\centering
\includegraphics[width=0.7\linewidth]{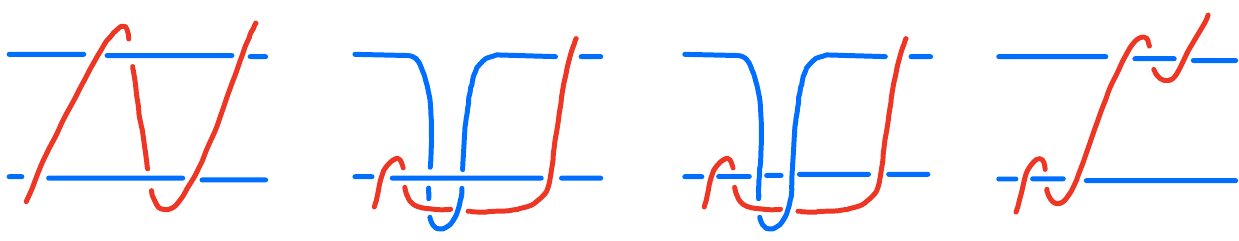}
\end{figure}

\noindent Using this movement, one can exchange the positions of two consecutive crossings between~$K_2$ and~$L_1$.
Since elementary transpositions generate all permutations, the crossings of~$K_1$ with~$L_2$ can 
be permutated arbitrarily. Thus, we can order these crossings according to our fixed
ordering of the components of~$L_1$, leading to the final normal form illustrated in Figure~\ref{Normalform} (right). This normal form being uniquely determined up to isotopy by the matrix~$\left(\lk(K_1,K_2)\right)_{K_1\subset L_1,K_2\subset L_2}$, this concludes the proof.
\end{proof}

\begin{figure}%normal form
\centering
\includegraphics[width=0.85\linewidth]{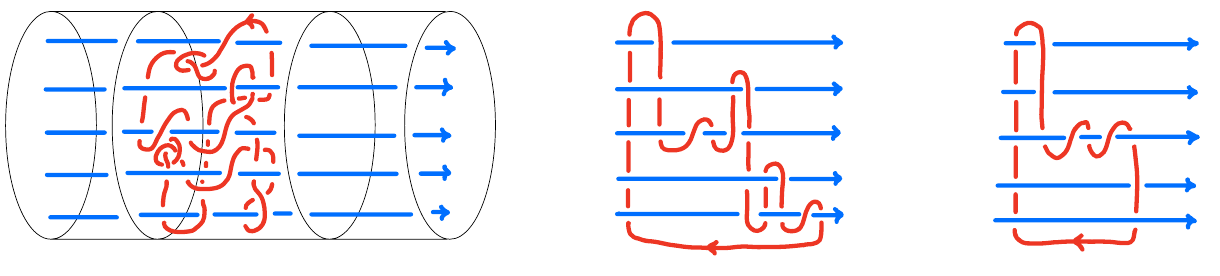}
\caption{The three steps in obtaining the normal form.}
\label{Normalform}
\end{figure}

\begin{proof}[Proof of Proposition~\ref{prop:lk}]
Let us fix an~$m$-colored link~$L$ and some subset~$E\subset\{\pm\}^m$ of even cardinality.
By Lemma~\ref{partition}, there exists a partition~$\{1,\dots, m\}=I_0\sqcup I_1$ such that~$\Arf_{\!E}(L)$
is invariant under crossing changes between components of colors belonging to the same subset.
\dc{If one of the subsets~$I_0$ or~$I_1$ is empty, then~$\Arf_{\!E}(L)$ does not depend on~$L$ and the 
proposition holds trivially. Otherwise,}
consider the~2-colored link~$\widetilde{L}$ obtained from~$L$ by identifying the colors belonging to the same subset of the partition.
We now know that the value of~$\Arf_{\!E}(L)$ is invariant under homotopy of the~2-colored link~$\widetilde{L}$.
By Lemma~\ref{lemma:classification}, the homotopy class of such a~2-colored link is determined by the
linking numbers of its components.
Since the linking numbers of the components of~$L$ and of~$\widetilde{L}$ coincide,
this concludes the proof.
\end{proof}

\appendix

\section{Generalized S-equivalence for totally connected C-complexes}
\label{sec:appendix}

\dc{Consider two totally connected C-complexes for two isotopic colored links.
The aim of this appendix is to give a detailed description of how the two sets of corresponding
generalized Seifert matrices are related. (See~\cite{Cim04,DMO21} for similar computations 
involving C-complexes that are not necessarily totally connected.)}

\medskip

\dc{We start with the proof of Lemma~\ref{S-equivalence}, whose statement we recall for the reader's convenience.}

\begin{lemma}
Let $F$ and $F'$ be two totally connected C-complexes for isotopic colored links. Then, there exist totally connected C-complexes $F=F^1,F^2,\dots,F^{n-1},F^n=F'$ such that for all~$1\le k <n$,~$F^{k+1}$ is obtained from $F^k$ by one of the following moves or its inverse:
\begin{itemize}
\item[(T0)] ambient isotopy;
\item[(T1)] handle attachment along one of the surfaces (see Figure~\ref{ST1-2}, top);
\item[(T2)] add a ribbon intersection and push along an arc (see Figure \ref{T23}, left);
\item[(T3)] pass through a clasp (see Figure \ref{T23}, right);
\item[(T4)] replace a push along an arc by a push along another arc (see Figure~\ref{T4}).
\end{itemize}
\end{lemma}

\begin{proof}
Let~$F=F_1\cup\dots\cup F_m$ and~$F'=F'_1\cup\dots\cup F'_m$ be two totally connected C-complexes for isotopic colored links.
By Lemma 4.2 of \cite{DMO21}, there exists a C-complex~$G=G_1\cup\dots\cup G_m$
(resp.~$G'=G'_1\cup\dots\cup G'_m$) obtained from~$F$ (resp.~$F'$) by a sequence of isotopies and handle attachments, such that~$G_i$ is isotopic to~$G_i'$ for all~$i$.
Since the C-complexes~$F$ and~$F'$ are assumed to be totally connected with connected surfaces,
so are all the C-complexes in these two sequences. Therefore, it can be assumed that~$F$ and~$F'$
are (totally connected) C-complexes with~$F_i$ and~$F'_i$ connected and isotopic for all~$i$.
Propositions~4.3 and~4.6  of~\cite{DMO21} then imply that there exists a sequence of C-complexes $F=F^1,F^2, \dots,F^{2\ell}=F'$ such that for all~$1\le k<\ell$, the C-complex~$F^{2k}$ is obtained from~$F^{2k-1}$ by an ambient isotopy while~$F^{2k+1}$ is obtained from~$F^{2k}$ by a move~(T2), (T3), (T4), or its inverse.
The possible issue here is that the inverse of a move~(T2) or a move~(T3) might transform a totally connected
C-complex into a no longer totally connected one. However, the following procedure ensures that such a
move can be replaced by other ones that keep the C-complexes totally connected.

To describe it, we first fix some terminology and notation.
For all~$1\le k<\ell$, let us denote by~$\mathrm{T}_k$ the move performed on~$F^{2k}$ to obtain~$F^{2k+1}$, and by~$B_k$ a (small)~$3$-ball containing the support of this move.
Fixing~$1\le i\neq j\le m$, we shall say that a C-complex is~\emph{$(i,j)$-connected} if its surfaces of colors~$i$
and~$j$ intersect. If all the C-complexes~$F^1,F^2, \dots,F^{2\ell}$ are~$(i,j)$-connected, then there is nothing to do.
Otherwise, let~$\kappa\in\{1,\dots,\ell-1\}$ denote the first time when~$(i,j)$-connectivity is lost; in other words, let us suppose that~$F^{2k-1}\simeq F^{2k}$ is~$(i,j)$-connected for~$1\le k\le\kappa$, but that~$F^{2\kappa+1}_i \cap F^{2\kappa+1}_j$ is empty.
Then, fix a simple arc~$\gamma_\kappa$ linking a point of~$\partial F^{2\kappa}_i \backslash B_{\kappa}$ and a point of~$\partial F^{2\kappa}_j \backslash B_{\kappa}$, with the interior of~$\gamma_\kappa$ disjoint from~$F^{2\kappa}\cup B_{\kappa}$. As illustrated in Figure~\ref{connection}, perform a ``connection'' via a move (T2) along~$\gamma_\kappa$, a transformation whose support lies in a small tubular neihbourhood~$\Gamma_\kappa$ of~$\gamma_\kappa$.
Performing the move~(T$_\kappa$) on the result produces an~$(i,j)$-connected C-complex~$(F^{2\kappa+1})'$,
and applying the ambient isotopy~$F^{2\kappa+1}\simeq F^{2\kappa+2}$ to~$(F^{2\kappa+1})'$ yields a C-complex~$(F^{2\kappa+2})'$ which is obviously~$(i,j)$-connected as well.
Since the image of~$\Gamma_\kappa$ by the previous isotopy might intersect~$B_{\kappa+1}$, let us now first
perform a new connection along an arc~$\gamma_{\kappa+1}$ linking a point of~$(F^{2\kappa+2})'_i \backslash B_{\kappa+1}$ with a point of~$(F^{2\kappa+2})'_j \backslash B_{\kappa+1}$, with the interior of~$\gamma_{\kappa+1}$ disjoint from~$(F^{2\kappa+2})' \cup B_{\kappa+1}$.
Then, remove the first connection (via the inverse of a move (T2) and an isotopy), and finally perform the move ($\mathrm{T}_{\kappa+1}$) on the result, yielding an~$(i,j)$-connected C-complex~$(F^{2\kappa+3})'$.
Iterating this procedure produces a sequence of~$(i,j)$-connected C-complexes from~$F$ to a C-complex obtained from~$F'$ by a move (T2) involving the surfaces~$F_i'$ and~$F_j'$. This connection can be removed, resulting
in a sequence of~$(i,j)$-connected C-complexes from~$F$ to~$F'$.
Applying this to all pairs of colors yields the desired result.
\end{proof}

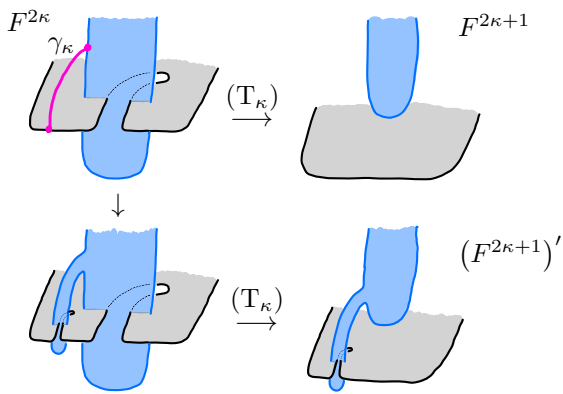
\begin{figure}[h]%connection
\centering
\begin{overpic}[width=0.4\linewidth]{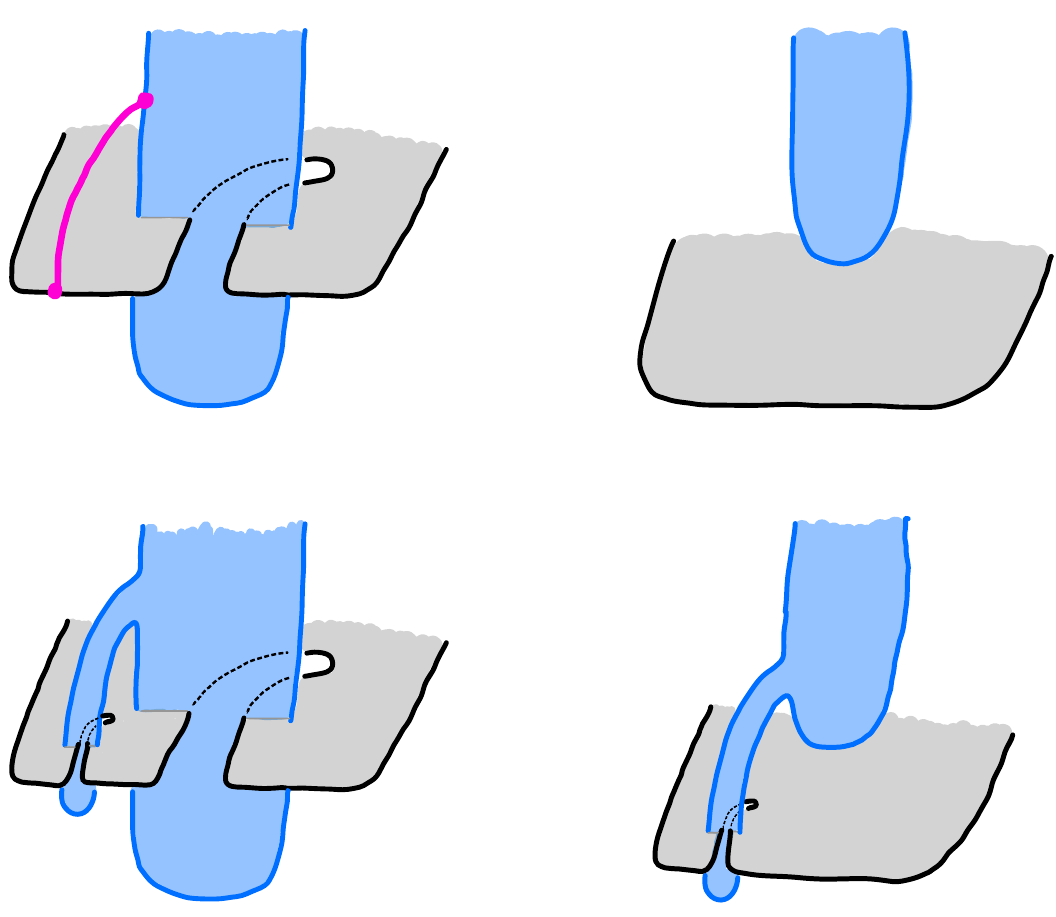}
\put(-4,81){$F^{2\kappa}$}
\put(95,80){$F^{2\kappa+1}$}
\put(95,30){$\left(F^{2\kappa+1}\right)'$}
\put(44,65){$\left(\mathrm{T}_\kappa\right)$}
\put(45,60){$\longrightarrow$}
\put(45,20){$\left(\mathrm{T}_\kappa\right)$}
\put(46,15){$\longrightarrow$}
\put(5,77){$\gamma_\kappa$}
\put(19,41){$\downarrow$}

\end{overpic}
\caption{A connection via a move (T2) along an arc.}
\label{connection}
\end{figure}

\dc{We now exhibit how moves~(T1) to~(T4) affect the corresponding generalized Seifert matrices.
The result is a extended notion of S-equivalence valid for generalized Seifert matrices associated to
totally connected C-complexes.} This requires some notation.

\begin{notation}
\label{notation}
Given two signs~$\varepsilon,\sigma\in\{\pm\}$, we write~$\delta(\varepsilon,\sigma)=1$ if~$\varepsilon=\sigma$, and~$\delta(\varepsilon,\sigma)=0$ otherwise.
Also, for~$z$ an element of a basis of a free abelian group~$H\simeq\Z^d$, we let~$\chi(z)$ denote the element in~$\Z^d$ with the~$z$-coordinate equal to~$1$ and all the others equal to~$0$.
\end{notation}

(T1) Let~$F'$ be a C-complex obtained from a C-complex~$F$ by a movement~(T1)
on the $\sigma_i$-side of $F_i$,
for some fixed color~$i\in\{1,\dots,m\}$ and sign~$\sigma_i=\pm$. Since~$F_i$ is connected,
we have~$H_1(F')=H_1(F)\oplus \mathbb{Z}x' \oplus \mathbb{Z}y'$, with~$x'$ and~$y'$ the cycles represented in the top part of Figure~\ref{ST1-2}.
 In the corresponding basis of~$H_1(F')$, the generalized Seifert matrix~$A^\varepsilon_{F'}$ is obtained from $A^\varepsilon_F$ via
\begin{equation}
\label{eq:T1}
A^\varepsilon_{F'}=\begin{pmatrix}
A^\varepsilon_F & 0 & \xi(\varepsilon_i) \\
0 & 0 & -\delta(\varepsilon_i,\sigma_i) \\
\xi(-\varepsilon_i)^{\T} & -\delta(\varepsilon_i,-\sigma_i) & \ell
\end{pmatrix}\,,
\end{equation}
with~$\ell\in\Z$ and~$\xi(\varepsilon_i)$ some integral vector.

\medskip

(T2) Let us now fix two colors~$i,j\in\{1,\ldots,m\}$ and two signs~$\sigma_i, \sigma_j$. Let~$F'$ be the (totally connected) C-complex obtained from a totally connected C-complex~$F$ by a movement~(T2) between the surfaces~$F_i$ and~$F_j$, as illustrated in the bottom part of Figure~\ref{ST1-2}. Since~$F$ is totally connected,
we have~$H_1(F')=H_1(F)\oplus \mathbb{Z}x' \oplus \mathbb{Z}y'$, with~$x',y'\subset F'_i\cup F'_j$ the cycles represented in Figure~\ref{ST1-2}.
In the corresponding basis of~$H_1(F')$, the generalized Seifert matrix~$A^\varepsilon_{F'}$ is obtained from~$A^\varepsilon_F$ via
\[
A^\varepsilon_{F'}=\begin{pmatrix}
A^\varepsilon_F & 0 & \xi(\varepsilon_i,\varepsilon_j)\\
0 & 0 & -\delta(\varepsilon_i,\sigma_i)\delta(\varepsilon_j,\sigma_j)\\
\xi(-\varepsilon_i,-\varepsilon_j)^{\T} & -\delta(\varepsilon_i,-\sigma_i)\delta(\varepsilon_j,-\sigma_j) & \ell(\varepsilon_i,\varepsilon_j)
\end{pmatrix}\,,
\]
with~$\ell(\varepsilon_i,\varepsilon_j)=\ell(-\varepsilon_i,-\varepsilon_j)\in\Z$ and~$\xi(\varepsilon_i,\varepsilon_j)$ some integral vector which only depends on~$\varepsilon_i$ and~$\varepsilon_j$.

\begin{figure}%movements T1 & T2
\centering
\begin{overpic}[width=0.6\linewidth]{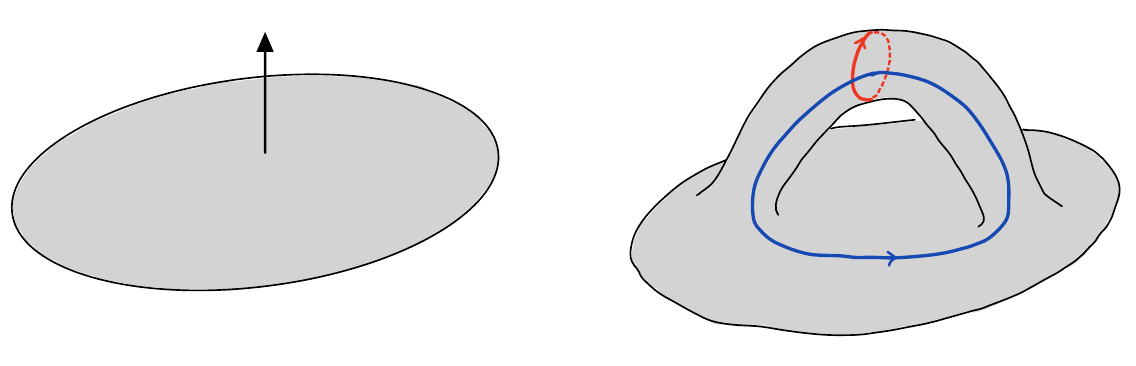}
\put(1,25){$F$}
\put(95,25){$F'$}
\put(18,29){$\sigma_i$}
\put(74,31){$x'$}
\put(68,6){$y'$}
\put(48,13){$\longrightarrow$}
\end{overpic}
\begin{overpic}[width=0.6\linewidth]{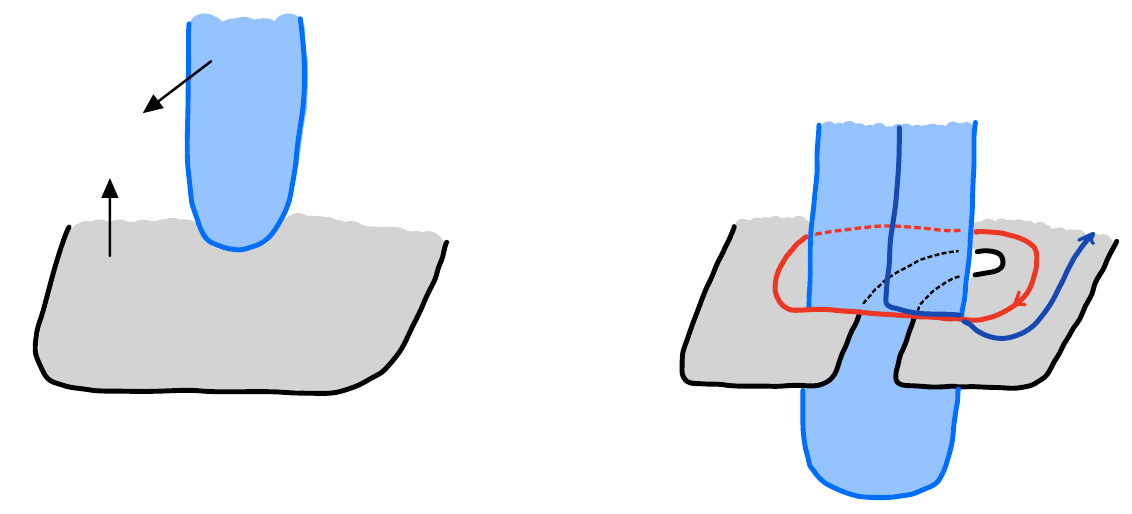}
\put(1,33){$F$}
\put(95,30){$F'$}
\put(9,37){$\sigma_i$}
\put(9,31){$\sigma_j$}
\put(66,23){$x'$}
\put(79,36){$y'$}
\put(48,15){$\longrightarrow$}
\end{overpic}
\caption{For the moves~(T1) and~(T2), a basis of~$H_1(F')$ is obtained by adding the cycles~$x'$ and~$y'$ to a basis of~$H_1(F)$.}
\label{ST1-2}
\end{figure}
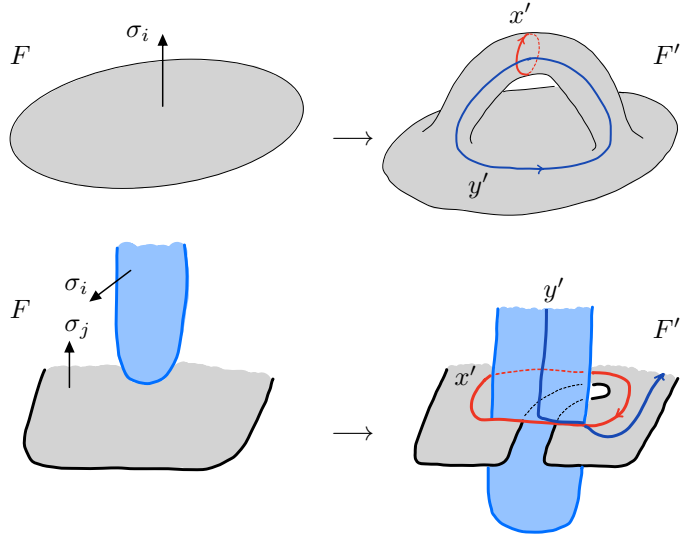

\medskip

(T3) Fix three colors~$i,j,k\in\{1,\ldots,m\}$ and three signs~$\sigma_i, \sigma_j, \sigma_k$. Let $F'$ be the totally connected C-complex obtained from the totally connected C-complex~$F$ by a movement~(T3) between the surfaces~$F_i$,~$F_j$ and~$F_k$ as illustrated in Figure~\ref{ST3}.
The homology groups~$H_1(F)$ and~$H_1(F')$ are naturally isomorphic,
with the cycle~$x$ in~$F$ mapped to the cycle~$x'$ in~$F'$, see Figure~\ref{ST3}.
The total connectivity of~$F$ (resp.~$F'$) guarantees the existence of a cycle~$y\subset F_i\cup F_k$ (resp.~$y'\subset F'_i\cup F'_j\cup F'_k$)
as illustrated in Figure~\ref{ST3},
with the natural isomorphism~$H_1(F)\simeq H_1(F')$ mapping~$y$ to~$y'$.
For the same reason, there exist cycles~$z\subset F$ and~$z'\subset F'$ as depicted in Figure~\ref{ST3};
moreover, one can fix a basis of~$H_1(F)$ (resp.~$H_1(F')$) having a unique cycle~$z$ (resp.~$z'$) of this form.
With respect to a basis of~$H_1(F)$ (resp.~$H_1(F')$) having~$x,y$ (resp.~$x',y'$) as last two elements,
the generalized Seifert matrices then have the form
\[
A^\varepsilon_{F}=\begin{pmatrix}
A(\varepsilon) & -\delta(\varepsilon_i,-\sigma_i)\chi(z) & \xi(\varepsilon_i,\varepsilon_k)\\
-\delta(\varepsilon_i,\sigma_i)\chi(z)^{\T} & 0 & -\delta(\varepsilon_i,\sigma_i)\delta(\varepsilon_k,\sigma_k)\\
\xi(-\varepsilon_i,-\varepsilon_k)^{\T} & -\delta(\varepsilon_i,-\sigma_i)\delta(\varepsilon_k,-\sigma_k) & \ell(\varepsilon_i,\varepsilon_k)
\end{pmatrix}
\]
and
\[
A^\varepsilon_{F'}=\begin{pmatrix}
A(\varepsilon) & -\delta(\varepsilon_j,-\sigma_j)\chi(z') & \xi(\varepsilon_i,\varepsilon_j, \varepsilon_k)\\
-\delta(\varepsilon_j,\sigma_j)\chi(z)^{\T} & 0 & -\delta(\varepsilon_j,\sigma_j)\delta(\varepsilon_k,\sigma_k)\\
\xi(-\varepsilon_i,-\varepsilon_j,-\varepsilon_k)^{\T} & -\delta(\varepsilon_j,-\sigma_j)\delta(\varepsilon_k,-\sigma_k) & \ell(\varepsilon_i, \varepsilon_j, \varepsilon_k)
\end{pmatrix}\,,
\]
with~$\ell(\varepsilon_i, \varepsilon_j, \varepsilon_k)\in\mathbb{Z}$,~$\xi(\varepsilon_i, \varepsilon_j, \varepsilon_k)$ an integral vector and~$A(\varepsilon)$ an integral matrix.

\begin{figure}%movement T3
\centering
\begin{overpic}[width=0.6\linewidth]{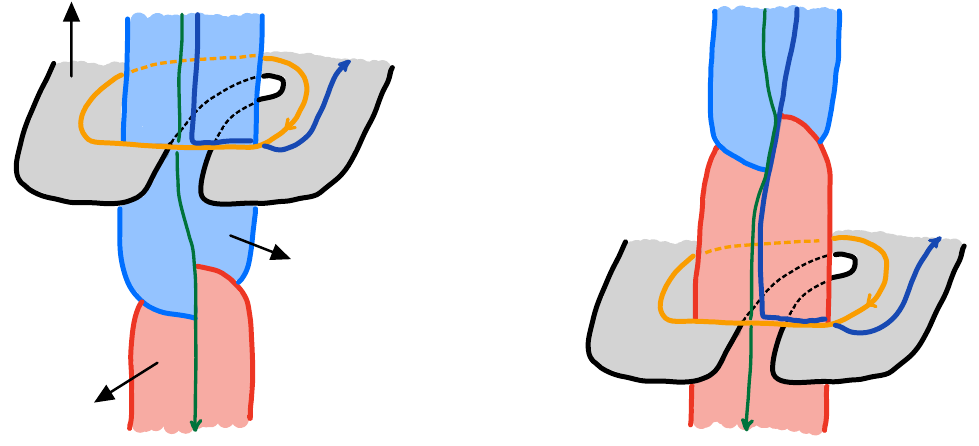}
\put(-2,42){$F$}
\put(65,35){$F'$}
\put(30,17){$\sigma_i$}
\put(6,47){$\sigma_j$}
\put(5,4){$\sigma_k$}
\put(5,28){$x$}
\put(35,41){$y$}
\put(19,-2){$z$}
\put(65,9){$x'$}
\put(95,23){$y'$}
\put(75,-2){$z'$}
\put(46,22){$\longrightarrow$}
\end{overpic}
\caption{The surfaces and cycles involved in the move~(T3).}
\label{ST3}
\end{figure}
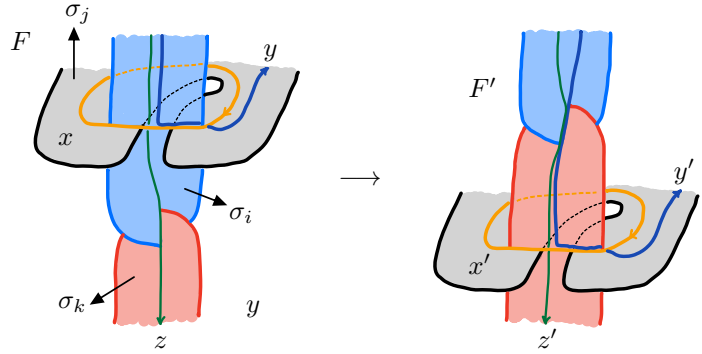

\medskip

(T4) Finally, fix two colors~$i,j\in\{1,\ldots,m\}$ and two signs~$\sigma_i, \sigma_j$. Let $F'$ be the C-complex obtained from a C-complex~$F$ by a movement~(T4) between the surfaces~$F_i$ and~$F_j$ as in Figure~\ref{ST4}.
Via conjugation by a move~(T2) along~$F_i\cup F_j$ and by moves~(T1) along~$F_i$ and along~$F_j$,
one can fix a basis of~$H_1(F)$ (resp.~$H_1(F')$) containing exactly three cycles~$x,y,z\subset F_i\cup F_j$ (resp.~$x',y',z'\subset F'_i\cup F_j'$) as in Figure~\ref{ST4}.
There is a natural isomorphism~$H_1(F)\simeq H_1(F')$ mapping the classes of~$x,y,z$ to the classes of~$x',y',z'$, respectively.
% (and the class of any cycle of the form of~$w$ to the corresponding cycle denoted by~$w'$ in Figure~\ref{ST4}.
Expressed in a basis of~$H_1(F)$ (resp.~$H_1(F')$) as above and having~$x,y$ (resp.~$x',y'$) as last two elements,
and using the aforementioned isomorphism, the generalized Seifert matrices have the form
\begin{equation}
\label{eq:T4}
A^\varepsilon_{F}=\begin{pmatrix}
A(\varepsilon) & \delta(\varepsilon_i,-\sigma_i)\chi(z) & \xi(\varepsilon_i,\varepsilon_j)\\
\delta(\varepsilon_i,\sigma_i)\chi(z)^{\T} & 0 & \delta(\varepsilon_i,\sigma_i)\delta(\varepsilon_j,\sigma_j)\\
\xi(-\varepsilon_i,-\varepsilon_j)^{\T} & \delta(\varepsilon_i,-\sigma_i)\delta(\varepsilon_j,-\sigma_j) & \ell(\varepsilon_i,\varepsilon_j)
\end{pmatrix}
\end{equation}
and
\begin{equation}
\label{eq:T4'}
A^\varepsilon_{F'}=\begin{pmatrix}
A(\varepsilon) & \delta(\varepsilon_i,\mp\sigma_i)\chi(z') & \xi(\varepsilon_i,\varepsilon_j)+n\chi(z')\\
\delta(\varepsilon_i,\pm\sigma_i)\chi(z')^{\T} & 0 & \delta(\varepsilon_i,\pm\sigma_i)\delta(\varepsilon_j,\sigma_j)\\
\xi(-\varepsilon_i,-\varepsilon_j)^{\T}+n\chi(z')^{\T} & \delta(\varepsilon_i,\mp\sigma_i)\delta(\varepsilon_j,-\sigma_j) & \ell(\varepsilon_i, \varepsilon_j)+n
\end{pmatrix}\,,
\end{equation}
with~$\ell(\varepsilon_i, \varepsilon_j)\in\mathbb{Z}$,~$\xi(\varepsilon_i,\varepsilon_j)$ an integral vector,~$A(\varepsilon)$ an integral matrix, and~$n\in\mathbb{Z}$ independent of~$\varepsilon$.
This integer is the winding number of the path around the ribbon; in particular, it vanishes in the case of~$(\mathrm{T4_a})$ illustrated in the left of Figure~\ref{T4}.
The sign in front of~$\sigma_i$ depends on how the arc traverses~$F_i$, the two different cases being illustrated
in Figure~\ref{ST4}.

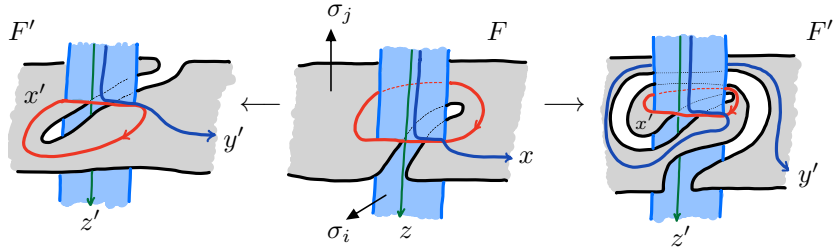
\begin{figure}%movement T4
\centering
\begin{overpic}[width=0.7\linewidth]{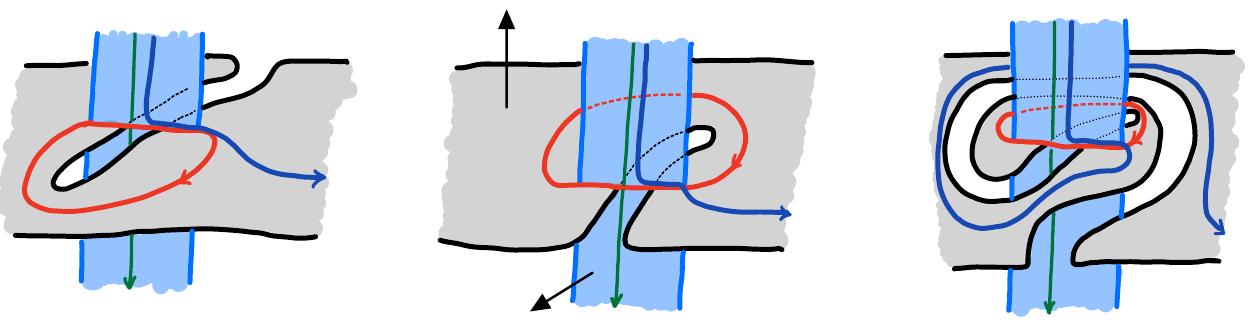}
\put(60,24){$F$}
\put(0,24){$F'$}
\put(100,24){$F'$}
\put(40,27){$\sigma_j$}
\put(40,-1){$\sigma_i$}
\put(64,8.5){$x$}
\put(35,41){$y$}
\put(49,-1){$z$}
\put(2,16){$x'$}
\put(27,10.5){$y'$}
\put(9,0){$z'$}
\put(78.6,12.8){\scriptsize{$x'$}}
\put(99,6){$y'$}
\put(83,-2){$z'$}
\put(67,15){$\longrightarrow$}
\put(29,15){$\longleftarrow$}
\end{overpic}
\caption{The surfaces and cycles involved in two possible versions of the move~(T4).}
\label{ST4}
\end{figure}

\medskip

In conclusion, two sets of generalized Seifert matrices corresponding to totally connected C-complexes for isotopic colored links are related
by simultaneous congruences, and the four transformations described above.

\bibliographystyle{plain}
\bibliography{Arf}

\end{document}